\documentclass{amsart}

\usepackage{graphicx}
\usepackage{graphics}
\usepackage{amssymb, amsmath, amsthm}
\usepackage{geometry}
\usepackage{lipsum}
\usepackage{enumerate}
\usepackage{dsfont}
\usepackage{mathrsfs}
\usepackage{xcolor}
\usepackage{multicol}
\usepackage{dblfloatfix}
\input xy
\xyoption{all}

\usepackage{hyperref}

\newtheorem{theorem}{Theorem}[section]
\newtheorem{lemma}[theorem]{Lemma}
\newtheorem{corollary}[theorem]{Corollary}
\newtheorem{proposition}[theorem]{Proposition}

\newtheorem{innercustomthm}{Theorem}
\newenvironment{customthm}[1]
  {\renewcommand\theinnercustomthm{#1}\innercustomthm}
  {\endinnercustomthm}

\newtheorem{innercustomcor}{Corollary}
\newenvironment{customcor}[1]
  {\renewcommand\theinnercustomcor{#1}\innercustomcor}
  {\endinnercustomcor}

\theoremstyle{definition}
\newtheorem{definition}[theorem]{Definition}
\newtheorem{example}[theorem]{Example}

\theoremstyle{remark}
\newtheorem{remark}[theorem]{Remark}

\numberwithin{equation}{section}

\newcommand{\e}{\mathbf{e}}

\newcommand{\mc}{\mathcal}

\newcommand{\mb}{\mathbb}
\newcommand{\mf}{\mathfrak}
\newcommand{\on}{\operatorname}
\newcommand{\opn}{\operatorname}

\newcommand{\x}{\mathbf{x}}

\renewcommand{\hat}{\widehat}
\newcommand{\p}{\mathfrak{p}}

\newcommand{\s}{{\scriptscriptstyle \leqslant }}
\renewcommand{\leq}{\leqslant}
\newcommand{\n}{{\scriptscriptstyle\lhd}}

\newcommand{\ds}{\displaystyle}

\newcommand{\Spec}{\operatorname{Spec}}

\begin{document}
	\title{On the degree of polynomial subgroup growth of nilpotent groups}

	\author{D. Sulca}
	\address{Facultad de Matem\'aticas, Astronom\'ia y F\'isica, Universidad Nacional de C\'ordoba, Ciudad Universitaria, C\'ordoba X5000HUA, Argentina}
	
	\email{diego.a.sulca@unc.edu.ar}
	
	
	
	
	
	
	
	
	\keywords{Subgroup growth, zeta functions of group and rings}
	\subjclass[2010]{11M41; 20E07}
	
	\begin{abstract}
	Let $N$ be a finitely generated nilpotent group. The subgroup zeta function $\zeta_N^\s(s)$ and the normal zeta function $\zeta_N^\n(s)$ of $N$ are Dirichlet series enumerating the finite index subgroups or the finite index normal subgroups of $N$. 
We present results about their abscissae of convergence $\alpha_N^\s$ and $\alpha_N^\n$, also known as the degrees of polynomial subgroup growth and polynomial normal subgroup growth of $N$, respectively. 

	We first prove some upper bounds for the functions $N\mapsto \alpha_N^\s$ and $N\mapsto\alpha_N^\n$ when restricted to the class of torsion-free nilpotent groups of a fixed Hirsch length. 
We then show that if two finitely generated  nilpotent groups have  isomorphic $\mb{C}$-Mal'cev completions, then their subgroup (resp.\ normal) zeta functions have the same abscissa of convergence. This follows, via the Mal'cev correspondence, from a similar result that we establish for zeta functions of rings. 
This result is obtained by proving that the abscissa of convergence of an Euler product of certain Igusa-type local zeta functions introduced by du Sautoy and Grunewald remains invariant under base change. We also apply this methodology to formulate and prove a version of our result about nilpotent groups for virtually nilpotent groups.

		As a side application of our result about zeta functions of rings, we present a result concerning the distribution of orders in number fields.
\end{abstract}

	\maketitle	
\section{Introduction}
\noindent Let $G$ be a finitely generated group and let
\begin{align*}
a_n^\s(G):=|\{ H\leq G : [G:H]=n\}|,\quad a_n^\n(G):=|\{H\lhd G  : [G:H]=n\}|
\end{align*}
be the number of subgroups or normal subgroups of index $n$ in $G$. These numbers are finite and
the associated Dirichlet series
\begin{align*}
	\zeta_G^{\s}(s):=\sum_{n=1}^\infty\frac{a_n^\s(G)}{n^s}=\sum_{H\leq_f G} [G:H]^{-s},\quad\quad  \zeta_G^\n(s):=\sum_{n=1}^\infty \frac{a_n^\n(G)}{n^s}=\sum_{H\lhd_f G} [G:H]^{-s},
\end{align*}
are called {\em the subgroup zeta function} and {\em the normal zeta function} of $G$, respectively. The symbol $H\leq_f G$ (resp.\ $H\lhd_f G$) indicates that the summation is over all subgroups (resp.\ normal subgroups) $H$ of finite index in $G$. We write $\zeta_G^*(s)$ when we intend to address both types of zeta functions simultaneously. 
These zeta functions were introduced by Grunewald, Segal and Smith in the landmark paper \cite{GSS}. 
\begin{example}\label{zeta function of Z^d}
Let $h\in \mathbb{N}$ and let $\mathbb{Z}^h$ be the free abelian group of rank $h$. Then
\begin{align*}
	\zeta_{\mb{Z}^h}^*(s)=\zeta(s)\zeta(s-1)\cdots\zeta(s-h+1),
\end{align*}
where $\zeta(s):=\sum_{n=1}^\infty n^{-s}$ is the Riemann zeta function (cf.\ \cite[Proposition 1.1]{GSS} or \cite[Chapter 15]{LS}).
\end{example}

Observe that the subgroup zeta function of $\mb{Z}^h$ converges on a non-empty region of the complex plane, namely the region $\on{Re}(s)>h$. This is a characteristic property of groups of {\em polynomial subgroup growth} (PSG), i.e.\ groups $G$ for which the function $n\mapsto \sum_{i=1}^na_i^\s(G)$ is bounded by a polynomial function in $n$.  More generally, for $*\in\{\leq,\lhd \}$ we set
$$
\alpha_G^*:=\inf\left\{\alpha \ |\ \exists c>0\ \forall n: \sum_{i=1}^n a_i^*(G)\leq c n^\alpha\right\},
$$
where conventionally $\inf\emptyset=\infty$. When $\alpha_G^*<\infty$, we call this number {\em the degree of polynomial (normal) subgroup growth of $G$}.
If $a_n^*(G)\neq 0$ for infinitely many $n$, then $\alpha_G^*$ coincides with the abscissa of convergence of $\zeta_G^*(s)$, that is, $\zeta_G^*(s)$ defines an analytic function on the region $\on{Re}(s)>\alpha_G^*$ and diverges for any $s$ with $\on{Re}(s)<\alpha_G^*$. 
A related invariant for groups of polynomial subgroup growth (called {\em the degree} of the group) was introduced and studied by Shalev; see for instance the influential paper \cite{Sh}.
For other types of growth we refer to \cite{LS}. 

The finitely generated groups of polynomial subgroup growth have been characterized algebraically by Lubotzky, Mann and Segal in \cite{LMS}. We recall this characterization.
Note first that the subgroup growth (or the normal subgroup growth) of a group $G$ is the same as that of the quotient group $G/R(G)$, where $R(G):=\bigcap_{N\lhd_f G} N$ is the finite residual of $G$, so there is no loss of generality in assuming that the groups under consideration are residually finite, i.e.\ the finite residual is trivial.  
It is proven in \cite{LMS} that a finitely generated residually finite group has polynomial subgroup growth if and only if it is virtually soluble of finite rank. 

We shall only deal with groups of polynomial subgroup growth that are nilpotent or virtually nilpotent. For these groups we will make some observations about the behavior of the function $G\mapsto\alpha_G^*$. These observations are mainly corollaries of properties of certain Igusa-type zeta functions called {\em cone integrals} that arise in the analysis of $\zeta_G^*(s)$. 

\medskip

\subsection{Zeta functions of $\mf{T}$-groups}
Since their introduction, zeta functions of groups have been investigated mainly for {\em $\mathfrak{T}$-groups}, i.e.\ finitely generated torsion-free nilpotent groups.
If $N$  is a $\mf{T}$-group, then there is an Euler product decomposition
\begin{align}\label{Euler product}
	\zeta_N^*(s)=\prod_{p\ \textrm{prime}} \zeta_{N,p}^*(s)=\prod_{p\ \textrm{prime}} \zeta_{\hat{N}_p}^*(s),
\end{align}
where $\zeta_{N,p}^*(s):=\sum_{k=0}^\infty  a_{p^k}^*(N)p^{-ks}$ is {\em the local factor of $\zeta_N^*(s)$ at $p$}, and $\hat{N}_p$ denotes the pro-$p$ completion of $N$. 
In addition,  $\zeta_{N,p}^*(s)$ is a rational function in $p^{-s}$ \cite{GSS}.

The study of zeta functions of $\mf{T}$-groups led to the consideration of {\em zeta functions of rings}, which we recall below.
The book \cite{dSW} collects comprehensive information about the first stage of the theory of zeta functions of groups and rings. The survey \cite{V2} exposes new developments. Let us also mention  \cite{CSV}, \cite{LeeVoll2020}, \cite{Rossmann2018} and
Rossmann's computer-algebra
package ZETA \cite{Rossmann2017} (which effectively computes, among other things, many (normal) zeta functions of nilpotent groups of
moderate Hirsch length) just to illustrate the current activity on the subject. 
We summarize some outstanding analytic properties of zeta functions of $\mf{T}$-groups obtained by du Sautoy and Grunewald.
\begin{theorem}[\cite{dSG}]\label{results of dSG}
Let $N$ be an infinite $\mf{T}$-group.
\begin{enumerate}
\item  $\alpha_N^*$ is a rational number and there exists $\delta>0$ such that $\zeta_N^*(s)$ can be meromorphically continued to the region $\on{Re}(s)>\alpha_N^*-\delta$. The continued function is holomorphic on the line $\on{Re}(s)=\alpha_N^*$ except for a pole at $s=\alpha_N^*$. 
\item If $b_N^*$ is the order of the pole of the continued function and $g_N^*(s)$ denotes the continuation of $(s-\alpha_N^*)^{b_N^*}\zeta_N^*(s)$, then
\begin{align*}
\sum_{i=1}^n a_i^*(N)\ \sim \  \frac{g_N^*(\alpha_N^*)}{\alpha_N^*\cdot (b_{N}^*-1)!}\cdot  n^{\alpha_N^*}(\log n)^{b_N^*-1}
\end{align*}
\end{enumerate}
where we write $f(n)\sim g(n)$ if $\lim_{n\to\infty} f(n)/g(n)=1$.
\end{theorem}

One natural problem is to relate $\alpha_N^*,b_N^*\in\mathbb{R}$ to structural information about $N$. This was posed as Problem 1.1 in \cite{dSG1} and remains open in general.   Example \ref{zeta function of Z^d} shows that if $N$ is a free abelian group of rank $h\geq 1$, then $\alpha_N^*=h$ and $b_N^*=1$.
However, if $N$ is a non-abelian $\mf{T}$-group, the computation of $\alpha_N^*$ is already a challenge.
The values of $\alpha_N^*$ and $b_N^*$ for various $\mf{T}$-groups are collected in \cite{dSW}. 
In \cite[Section 6.2]{Rossmann20172}, Rossamnn computes $\alpha_N^\n$ when $N$ is a $\mf{T}$-group of maximal nilpotency class. 
As an example of computation of $\alpha_N^\n$, we mention the following remarkable calculation by  Voll, which was obtained by purely combinatorial methods.
\begin{example}[\cite{V0}]
	Let $F=F_{2,d}$ be the free nilpotent group of class 2 on $d$ generators ($d\geq 2$). Then
	$$
	\alpha_F^\n=\max\left\{d,\frac{(\binom{d}{2}-j)(d+j)+1}{\binom{d+1}{2}-j}\ \left| \  j=1,\ldots,\binom{d}{2}-1\right. \right\}. 
	$$
\end{example} 
\begin{remark}
The proof of Theorem 1.2(1) given in \cite{dSG} expresses $\alpha_N^*$ in terms of some numerical data associated to a principalization of an ideal of polynomials over $\mathbb{Q}$ obtained from $N$ (see Section \ref{Integrales conicas} for a review of this). Even though nowadays there are algorithmic resolutions of singularities and principalizations of ideals, they are impractical in this context since the ideals obtained from $N$ are very complicated (several polynomials in several variables), even for quite simple $N$.
\end{remark} 
\medskip

We mentioned that if $N$ is the free abelian group of rank $h$, then $\alpha_{N}^\s=\alpha_N^\n=h$. An extension of the notion of rank to the class of $\mf{T}$-group, or more generally, to the class of polycyclic groups, is the notion of Hirsch length. 
If $N$ is a polycyclic group, its Hirsch length is the number of infinite factors in a subnormal series with cyclic factors, and it is denoted by $h(N)$. 
Now, if $N$ is a non-abelian $\mf{T}$-group, then a simple argument shows that $h(N^{ab})\leq \alpha_N^\n\leq\alpha_N^\s\leq h(N)$, where $N^{ab}$ denotes the abelianization of $N$ (see \cite[Proposition 1]{GSS}).
There are better bounds for $\alpha_N^\n$ when $N$ is a non-abelian $\mf{T}$-group of nilpotency class 2; cf.\ 
\cite{PAA} and \cite[Proposition 6.3]{GSS}. 
We also mention the lower bound  $\frac{1}{6}h(N)\leq \alpha_N^\s$ (\cite[Theorem 5.6.6]{LS}), which actually holds for a larger class of groups including the polycyclic groups.
All these results are useful for the following problem: given $h>2$, describe the set $S_h^*$ of possible values for $\alpha_N^*$ as $N$ ranges over the non-abelian $\mf{T}$-groups of Hirsch length $h$. According to \cite[Proposition 1.1]{dSG1} and Theorem \ref{results of dSG}(a), $S_h^*$ is a finite subset of $[2,h]\cap\mb{Q}$. In particular, one can ask what is the maximum of $S_h^*$. Our first result gives a partial answer to this question.
\begin{customthm}{A}\label{main theorem:upper bound for a(N)}
Let $N$ be a non-abelian $\mf{T}$-group of Hirsch length $h$ and nilpotency class $c$.
Then $\alpha_N^\s\leq h-\frac{1}{2}$ if $c=2$  and $\alpha_N^\s\leq h-\frac{1}{c-1}$ if $c>2$. Also $\alpha_N^\n\leq h-1$.

In particular, given $h\in\mathbb{N}$, the only $\mf{T}$-group $N$ of Hirsch length $h$ with $\alpha_N^\s=h$ is the abelian group $\mathbb{Z}^h$. 
\end{customthm}
\begin{remark}
The bound $\alpha_N^\n\leq h-1$ is optimal. Indeed, choose integers $r\geq 0$ and $m>0$ such that $2m+1+r=h$, and let $G(m,r)$ be the product of $\mathbb{Z}^r$ with a central product of $m$ copies of the discrete Heisenberg group $H_3(\mathbb{Z})$. Then $h(G(m,r)^{ab})=h-1\leq \alpha_{G(m,r)}^\n\leq h-1$, which gives $\alpha_{G(m,r)}^\n=h-1$. 
In contrast, by our method in the proof of Theorem \ref{main theorem:upper bound for a(N)}, it seems that our bound for $\alpha_N^\s$ is far from being optimal. The few examples where $\alpha_N^\s$ has been computed show that $\alpha_N^\s\leq h(N)-1$ if $N$ is a non-abelian $\mf{T}$-group.
\end{remark}

\medskip 

We return with the notation of Theorem \ref{results of dSG}. In trying to understand which structural data of $N$ is reflected in $\alpha_N^*$ and $b_N^*$
it is natural to ask, given two $\mf{T}$-groups $N_1$ and $N_2$,  when $\alpha^*_{N_1}=\alpha_{N_2}^*$, and if this is so, when $b_{N_1}^*=b_{N_2}^*$.
According to \cite[Proposition 3]{GSS}, if $N_1$ and $N_2$ are commensurable $\mf{T}$-groups (i.e.\ there exist finite index subgroups $H_1\leq N_1$ and $H_2\leq N_2$ such that $H_1\cong H_2$), then $\alpha_{N_1}^*=\alpha_{N_2}^*$, and moreover $b_{N_1}^*=b_{N_2}^*$ (cf.\ Proposition \ref{the Mal'cev correspondence}). 
We rephrase this fact.  Recall first that a $\mf{T}$-group can be embedded as an arithmetic group of a uniquely determined  unipotent group scheme over $\mb{Q}$. In addition, two $\mf{T}$-groups are commensurable if and only if they are isomorphic to arithmetic groups of the same unipotent group scheme over $\mb{Q}$. Thus, the fact that $\alpha_N^*$ and $b_N^*$ are commensurability-invariant can be restated as follows: 
{\em 
 If $N_1$ and $N_2$ are arithmetic groups of the same unipotent group scheme over $\mb{Q}$, then $\alpha_{N_1}^*=\alpha_{N_2}^*$ and $b_{N_1}^*=b_{N_2}^*$.}
Our next result is a partial generalization of this property.
\begin{customthm}{B}\label{main theorem}
Let $\mf{N}_1$ and $\mf{N}_2$ be unipotent group schemes over $\mb{Q}$, and let $N_i$ be an arithmetic subgroup of $\mf{N}_i$ for $i=1,2$. 
If $\mf{N}_1$ and $\mf{N}_2$ are isomorphic after base change with $\mb{C}$, then  $\alpha_{N_1}^*=\alpha_{N_2}^*$ for $*\in\{\leq,\lhd\}$.
\end{customthm}
\begin{remark}
In contrast, we expect that $b_{N_1}^*=b_{N_2}^*$ does not hold in general. Indeed, Remark \ref{main theorem does not hold for b} below shows that this equality fails for zeta functions of rings in general. However, the counter-example that we will show does not belong to the realm of nilpotent Lie rings (or groups).
\end{remark}

In other words, Theorem \ref{main theorem} says that the number $\alpha_N^*$ associated to a $\mf{T}$-group $N$ is in fact an invariant of the $\mb{C}$-Mal'cev completion of $N$, or simply that it is a geometric invariant.
Similar conclusions in spirit can be deduced
 from the main result of \cite{DV} for the representation zeta function of $\mf{T}$-groups, and from the main results of \cite{AKOV} and \cite{LN} for the degree of polynomial representation growth and for the subgroup growth rate of arithmetic groups in simply connected absolutely simple group schemes over $\mathbb{Q}$. 
Observe finally that a result similar to Theorem \ref{main theorem} holds for the degree of polynomial word growth by the formula of Bass-Guivarc'h \cite{Bass}.

\begin{remark}
There is a known classification of nilpotent Lie algebras in dimension $\le 7$ over $\mb{C}$ (and not over $\mb{Q}$).
Hence Theorem \ref{main theorem} could be used to completely determine, possibly with
computer help, the number $\alpha_N^*$ for all $\tau$-groups $N$ of Hirsch length at most 7.
\end{remark}
 
\medskip

\subsection{Zeta functions of rings}
Theorem \ref{main theorem} is obtained as a consequence of a more general result, namely Theorem \ref{main theorem: general version}, which employs the concept of zeta functions of rings introduced in \cite[Section 3]{GSS}. By a {\em ring} we shall mean an abelian group $L$ with a bilinear map $L\times L\to L$ called {\em multiplication}, e.g.\ Lie rings and the commutative rings with identity. 
A {\em subring} of $L$ is an abelian subgroup $A$ that is closed under multiplication. To allow further applications, in case that $L$ is commutative ring with identity 1, we shall also require (as usual) that $1\in A$. A two-sided ideal of $L$ will be simply called {\em an ideal}.

Let $L$ be a ring additively isomorphic to $(\mb{Z}^h,+)$ for some $h\geq 1$. For each positive integer $n$, let $a_n^\s(L)$ and $a_n^\n(L)$ denote the number of subrings  or ideals of index $n$ in $L$. {\em The subring and the ideal zeta functions} of $L$ are the Dirichlet series $\zeta_L^\s(s):=\sum_{n=1}^\infty  a_n^\s(L)n^{-s}$ and $\zeta_L^\n(s):=\sum_{n=1}^\infty  a_n^\n(L)n^{-s}$, respectively. They have a factorization as an Euler product
\begin{align*}
	\zeta_{L}^*(s)=\prod_{p\ \textrm{prime}} \zeta_{L,p}^*(s),
\end{align*}  
where $\zeta_{L,p}^*(s):=\sum_{k=0}^\infty  a_{p^k}^*(L)p^{-ks}$. 

The zeta functions of an arbitrary ring of additive rank 2 were computed in \cite[Chapter 7]{Sn12}. The subring zeta function of an arbitrary Lie ring of additive rank 3 was computed in \cite{KlopschVoll2007}. Zeta functions of nilpotent Lie rings are essentially the same as zeta functions of $\mf{T}$-groups:
via the Mal'cev correspondence one can associate with a $\mf{T}$-group $N$ of Hirsch length $h$ a nilpotent Lie ring $L$ of additive rank $h$ (and viceversa), and it holds that $\zeta_L^*(s)$ and $\zeta_{N}^*(s)$ have the same local factors for almost all primes $p$ (see Proposition \ref{the Mal'cev correspondence}). 
Some information about zeta functions of soluble Lie rings of higher rank can be found in \cite[Chapter 3]{dSW}. 

Theorem \ref{results of dSG}, formulated there for zeta functions of $\mf{T}$-groups, was also proved for zeta functions of rings in \cite{dSG}. 
In particular, one can consider the pair $(\alpha_L^*,b_L^*)$, where $\alpha_L^*$ is the abscissa of convergence of $\zeta_L^*(s)$ and $b_L^*$ is the order of the pole of the continued function at $s=\alpha_L^*$. Our next result is: 
\begin{customthm}{C}\label{main theorem: general version}
Let $L_1$ and $L_2$ be two rings additively isomorphic to $(\mb{Z}^h,+)$. If $L_1\otimes_\mb{Z} \mb{C}\cong L_2\otimes_\mb{Z}\mb{C}$ as $\mb{C}$-algebras, then $\alpha_{L_1}^*=\alpha_{L_2}^*$ for $*\in\{\leq,\lhd\}$.
\end{customthm}
\begin{example}\label{ex: Heisenberg}
Let $\mc{H}:=\langle x,y,z: [x,y]=z,\ [x,z]=[y,z]=0\rangle$ be the discrete Heisenberg Lie ring. For each square free integer $k$ we consider $L_k:=\mc{H}\otimes_{\mb Z}\mc{O}_{\mb{Q}(\sqrt{k})}$, where
 $\mc{O}_{\mb{Q}(\sqrt{k})}$ denotes the ring of integers of the quadratic field $\mb{Q}(\sqrt{k})$. This is a 2-step nilpotent Lie ring of additive rank 6.
 If $k\neq k'$, then $L_k\otimes_\mb{Z}\mb{Q}$ and $L_{k'}\otimes_{\mb{Z}}\mb{Q}$ are not isomorphic as $\mb{Q}$-Lie algebras (cf. \cite[Proposition 3.2]{Lauret}). However, $L_k\otimes_\mb{Z}\mb{C}\cong \mc{H}^2\otimes_{\mb{Z}}\mb{C}$ as $\mb{C}$-Lie algebras (where $\mc{H}^2=\mc{H}\times\mc{H}$) and hence $\alpha_{L_k}^*=\alpha_{\mc{H}^2}^*$ for all $k$ as above.
	 The zeta functions of $\mc{H}^2$ were computed in \cite[Proposition 8.11]{GSS} and \cite{Taylor}, and it holds that $\alpha_{\mc{H}^2}^*=4$.
We can now use	  Theorem \ref{main theorem: general version} to conclude that $\alpha_{L_k}^*=4$ for all square-free integer $k$. 
	  
	  We remark that a formula for $\zeta_{L_k}^\n(s)$ was given in \cite[Corollary 8.2]{GSS}. More generally, for any number field $K$, the local factors of $\zeta_{\mc{H}\otimes_{\mb{Z}}\mc{O}_K}^\n(s)$ at almost all primes $p$ were computed in \cite{SV1} and \cite{SV2}. A further analysis on the Euler product of these local factors is required to compute the abscissa of convergence $\alpha_{\mc{H}\otimes_\mb{Z}\mc{O}_K}^\n$.
\end{example}

As a generalization of the above observation we have the following:
\begin{customcor}{D}\label{main result: corollary}
Let $L$ be a ring additively isomorphic to $(\mb{Z}^h,+)$, $K\supset\mb{Q}$ a number field of degree $d$ and $\mc{O}$ its ring of integers. 
Let $L_\mc{O}:=L\otimes_{\mb Z}\mc{O}$.
Then $\alpha_{L_\mc{O}}^*=\alpha_{L^d}^*$, where $L^d$ denotes the product of $d$ copies of the ring $L$.
\end{customcor}
\begin{proof}
Note that $L_\mc{O}\otimes_\mb{Z}\mb{C}=(L\otimes_{\mb{Z}}\mc{O})\otimes_{\mb{Z}}\mb{C}\cong L\otimes_{\mb Z}(\mc{O}\otimes_{\mb{Z}}\mb{C})\cong  L\otimes_{\mb{Z}}\mb{C}^d\cong L^d\otimes_{\mb{Z}}\mb{C}$ as $\mb{C}$-algebras. Hence $\alpha_{L_\mc{O}}^*=\alpha_{L^d}^*$ by Theorem \ref{main theorem: general version}. 
\end{proof}

\begin{remark}\label{main theorem does not hold for b}
Let $L=\mb{Z}$, viewed as commutative ring with identity. Then, $L_{\mb{Z}[i]}:=\mb{Z}\otimes_{\mb{Z}}\mb{Z}[i]$ is the ring of Gaussian integers and $\zeta_{L_{\mb{Z}[i]}}^\n(s)$ is the Dedekind zeta function $\zeta_{\mb{Q}(i)}(s)$ of $\mb{Q}(i)$. Hence $\alpha_{L_{\mb{Z}[i]}}^\n=1$ and $b_{L_{\mb{Z}[i]}}^\n=1$. Note also that the product ring $L^2$ has ideal zeta function $(\zeta(s))^2$, whence $\alpha_{L^2}^\n=1$ and $b_{L^2}^\n=2$. Since $L_{\mb{Z}[i]}\otimes_{\mb{Z}}\mb{C}\cong L^2\otimes_\mb{Z}\mb{C}$, we deduce that Theorem \ref{main theorem: general version} does not hold for $b_L^*$ in general.
\end{remark}

\begin{remark}
The computation of $\alpha_{L^d}^\s$ is in general a quite difficult task. For example, let $\mb{Z}_{\textrm{ring}}^d$ be the ring that is a product of $d$ copies of the ring of integers $\mb{Z}$. Then $\alpha_{\mb{Z}_{\textrm{ring}}^d}^\s=1$ for $d\leq 5$ (\cite[Theorem 6]{KMT}) while it is unknown for $d>5$.
The computation of $\alpha_{L^d}^\n$ might be also difficult if $L$ lacks an identity element. For instance, following up Example \ref{ex: Heisenberg} and Corollary \ref{main result: corollary} we find that $\alpha_{\mc{H}\otimes_\mb{Z}\mc{O}_K}^\n=\alpha_{\mc{H}^d}^\n$, where $d=[K:\mb{Q}]$. We have  $\alpha_{\mc{H}^d}^\n=2d$ if $d\leq 4$ (see \cite[Chapter 1]{dSW}), however nothing is known for $d>4$.
\end{remark}

\medskip 
Let us give an application of Corollary \ref{main result: corollary} to the distribution of orders in number fields.
Let $K$ be a number field and let $\mc{O}_K$ be its ring of integers. 
{\em An order} is a subring $\mc{O}$ of $\mc{O}_K$ with identity that is a $\mb{Z}$-module of rank $n$. 
Set
\begin{align*}
N_K(n):=|\{\mc{O}\subseteq\mc{O}_K\ |\ \on{disc}(\mc{O})\leq n\}|.
\end{align*}
The asymptotic behavior of $n\mapsto N_K(n)$ was studied in \cite{KMT}. It was shown, by an application of the results in \cite{dSG} (the version of Theorem \ref{results of dSG} for zeta functions of rings), that there exist $C_K\in\mathbb{R}$, $\alpha_K\in\mathbb{Q}$ and $\beta_K\in\mb{N}$ such that
\begin{align*}
N_K(n)\sim C_K n^{\alpha_K}(\log n)^{\beta_K-1}.
\end{align*}
It was also conjectured that the number $\alpha_K$ only depends on the degree $[K:\mb{Q}]$; see \cite[Conjecture 1]{KMT}. We now show that this is a special case of Corollary \ref{main result: corollary}.
\begin{customthm}{E}\label{main result: Conjecture}
Let $d:=[K:\mb{Q}]$, and let $\mb{Z}_{\textrm{ring}}^d$ denote the product of $d$ copies of the ring of integers $\mb{Z}$. Then $\alpha_K=\frac{1}{2} \alpha_{\mb{Z}_{\textrm{ring}}^d}^\s$. In particular, $\alpha_K$ only depends on the degree $[K:\mb{Q}]$.
\end{customthm}
\begin{proof}
 Let $\eta_K(s)=\ds\sum_{\mc{O}\ \textrm{order}} |\on{disc}\mc{O}|^{-s}$. Then $\eta_K(s)=|\on{disc}\mc{O}_K|^{-s}\zeta_{\mc{O}_K}^\s(2s)$, where $\mc{O}_K$ is seen as a ring with identity. Observe that $\alpha_K$ is the abscissa of convergence of $\eta_K(s)$, hence $\alpha_K=\frac{1}{2}\alpha_{\mc{O}_K}^\s$. Note also that Corollary \ref{main result: corollary} with $L=\mb{Z}$ (the ring of integers) yields $\alpha_{\mc{O}_K}^\s=\alpha_{\mb{Z}_{\textrm{ring}}^d}^\s$. Therefore, $\alpha_K=\frac{1}{2}\alpha_{\mb{Z}_{\textrm{ring}}^d}^\s$.
\end{proof}

\subsection{Methodology}
The idea behind the proof of Theorem \ref{main theorem: general version} is quite simple. If $L_1$ and $L_2$ are isomorphic after base change with $\mb{C}$, then $L_1\otimes_\mb{Z}K\cong L_2\otimes_\mb{Z}K$ for some number field $K$ (Lemma \ref{an algebraic geometry lemma}). 
Let $\mc{O}$ be the ring of integers of $K$. 
We associate to the $\mc{O}$-algebra ${L_i}_\mc{O}:=L_i\otimes_\mb{Z}\mc{O}$ the zeta functions $\zeta_{{L_i}_\mc{O}}^{\s_\mc{O}}(s)$ and $\zeta_{{L_i}_\mc{O}}^{\n_\mc{O}}(s)$ enumerating $\mc{O}$-subalgebras or $\mc{O}$-ideals of ${L_i}_\mc{O}$, and show that  $\zeta_{L_i}^*(s)$ and  $\zeta_{{L_i}_\mc{O}}^{*_\mc{O}}(s)$ have the same abscissa of convergence (Corollary \ref{invariance of a_L under base extension}). 
The proof of this fact makes use of the main tool of the paper: {\em cone integrals} (see Section \ref{Integrales conicas}).
 We show that for each $i=1,2$, $\zeta_{L_i}^*(s)$ and  $\zeta_{{L_i}_\mc{O}}^{*_{\mc{O}}}(s)$ are Euler products of cone integrals with the same cone integral data but over different fields, namely over $\mb{Q}$ and over $K$(Corollary \ref{existence of the cone integral data}), and hence they have the same abscissa of convergence.
This last assertion follows from Theorem \ref{basic properties of cone integrals}, which collects several properties about cone integrals. 
Finally, Theorem \ref{basic properties of cone integrals} will also enable us to conclude that $\zeta_{{L_1}_\mc{O}}^{*_\mc{O}}(s)$ and $\zeta_{{L_2}_\mc{O}}^{*_\mc{O}}(s)$ have the same abscissa of convergence (Corollary \ref{invariance of a_L and b_L by commmensurability}), from which Theorem \ref{main theorem: general version} follows.

Theorem \ref{basic properties of cone integrals} will be also used to formulate and prove a version of Theorem \ref{main theorem} for virtually nilpotent groups (Theorem \ref{main theorem for virtually nilpotent}). This will be possible as the zeta functions of these groups can be expressed as finite sums of series that are Euler products of cone integrals. We omit this version in the introduction and refer the reader to Section \ref{Sec: Main theorem for virtually nilpotent groups}.

\subsection{Organization and notation}
In Section \ref{Integrales conicas}, we recall the concept and some important properties of cone integrals over $\mb{Q}$ from \cite{dSG}, and we extend them for any number field. 
We use these results to prove Theorem \ref{main theorem: general version} in Section \ref{Sec: Proof of main theorem for rings} and to recall how this theorem implies Theorem \ref{main theorem}. In Section \ref{Sec: Proof of Theorem A}, we prove Theorem \ref{main theorem:upper bound for a(N)}. This section is, to a large extent, independent from the other ones. 
Finally, in Section \ref{Sec: Main theorem for virtually nilpotent groups} we formulate and prove Theorem \ref{main theorem for virtually nilpotent}, which is the analogous of Theorem \ref{main theorem} for virtually nilpotent groups.	

We write $\mathbb{N}$ for the set $\{1,2,\ldots\}$ and $\mathbb{N}_{0}$ for the set $\mb{N}\cup\{0\}$. 
We write $\mathbb{R}_{>0}$ for the set $\{s\in\mathbb{R}: s>0\}$ and $\mathbb{R}_{\geq 0}$ for the set $\{s\in\mathbb{R}:s\geq 0\}$. The notation $f(n)\sim g(n)$ means that $f(n)/g(n)$ tends to 1 as $n$ tends to infinity.

For a prime $p$, $\mb{Z}_p$ and $\mb{Q}_p$ denote the $p$-adic integers and the $p$-adic rationals, respectively. For a number field $K$ we denote by $\mc{O}_K$ its ring of integers. Given a maximal ideal $\p\subset\mc{O}:=\mc{O}_K$ we denote by $\mc{O}_\p$ and $K_\p$ the $\p$-adic completions of $\mc{O}$ and $K$.
Given $x\in {K}_\p$ we denote by $\on{ord}_\p(x)\in\mb{Z}\cup\{\infty\}$ its $\p$-adic valuation and write $|x|_\mathfrak{p}:={\mathbf{N}\p}^{-\on{ord}_\mathfrak{p}(x)}$ for its $\p$-adic norm, where $\mathbf{N}\p:=[\mc{O}:\p]$. 
\subsection*{Acknowledgment}
This work was supported by CONICET (Argentina).
Theorem \ref{main theorem} is part of my Ph.D.\  thesis. I would like to thank my supervisors Paulo Tirao and Karel Dekimpe for their guidance and encouragement. I would also like to thank Marcos Origlia for a careful reading of a previous draft of the paper and useful remarks.
Finally, I am especially grateful with the two referees, whose comments led to a significant improvement in the presentation and the content of the paper. For instance, the inclusion of Corollary \ref{main result: corollary} and its subsequent application Theorem \ref{main result: Conjecture} was suggested by them.

\section{A review of cone integrals}\label{Integrales conicas}
\noindent 
Cone integrals are a kind of $p$-adic integrals which can be seen as a generalization of Igusa local zeta functions.
They were introduced and studied by du Sautoy and Grunewald in \cite{dSG} under the assumption that the base field is $\mb{Q}$. 
Theorem \ref{basic properties of cone integrals} collects the main properties of cone integrals and at the same time it extends them to cone integrals over a general number field. The rest of the section is devoted to explaining how this general formulation follows essentially by the same arguments of \cite{dSG}.

\begin{definition}\label{definition of cone integrals}
Let $K$ be a number field and let $m$ be a positive integer. A finite collection $\mathcal{D}=(f_0,g_0,f_1,g_1,\ldots,f_l,g_l)$ of non-zero polynomials  in $K[x_1,\ldots,x_m]$ is called {\em a cone integral data} over $K$. Let $\mc{O}=\mc{O}_K$, and let $\mathfrak{p}\subset\mc{O}$ be a maximal ideal. Then the integral
	\begin{align*}
		Z_\mathcal{D}(s,\mathfrak{p})=\int_{\mathcal{M}(\mathcal{D},\mathfrak{p})}|f_0(\x)|_\mathfrak{p}^s|g_0(\x)|_\mathfrak{p}|d\x|_\p, 
	\end{align*}
	where
	\begin{align*}
		\mathcal{M}(\mathcal{D},\p)=\{\x\in {\mc O}_\mathfrak{p}^m\ |\ \on{ord}_\mathfrak{p}(f_i(\x))\leq \on{ord}_\mathfrak{p}(g_i(\x))\ \mbox{for}\ i=1,\ldots,l\}, 
	\end{align*}
and $|d\x|_\p=|dx_1\wedge \cdots \wedge dx_m|_\p$ is the normalized additive Haar measure on ${\mc O}_\p^m$,
is called	 \emph{a cone integral over $K$}, with cone integral data $\mc{D}$.  

It is easy to see that for each maximal ideal $\p\subset\mc{O}$, the integral $Z_\mc{D}(s,\p)$ can be expressed as a power series, say  $Z_\mathcal{D}(s,\mathfrak{p})=\sum_{i=0}^\infty a_{\mathfrak{p},i}\mathbf{N}\mathfrak{p}^{-is}$,  where each $a_{\p,i}$ is a non-negative rational number. We associate to $\mc{D}$ the Dirichlet series
\begin{align*}
Z_{\mc{D}}(s):=\prod_{\substack{\p\subset\mc{O}\ \textrm{maximal}\\ a_{\p,0}\neq 0}}a_{\p,0}^{-1} Z_{\mc{D}}(s,\p),
\end{align*}
and denote its abscissa of convergence by $\alpha_{\mc{D}}$. A function $Z(s)$ such that $Z(s)=Z_\mc{D}(s)$ is said to be {\em an Euler product of cone integrals over $K$} with cone integral data $\mc{D}$.
\end{definition}

\begin{theorem}\label{basic properties of cone integrals}
Assume that $Z_{\mc{D}}(s)$ is not the constant function. Then the following holds.
\begin{enumerate}
\item Each $Z_{\mc{D}}(s,\p)$ is a rational function in $\mathbf{N}\p^{-s}$ with rational coefficients.
\item $\alpha_\mc{D}$ is a rational number, and the abscissa of convergence of each $Z_\mc{D}(s,\p)$ is strictly to the left of $\alpha_\mc{D}$.
\item  There exists $\delta>0$ such that $Z_{\mc{D}}(s)$ has meromorphic continuation to the region $\{s\in\mb{C}: \on{Re}(s)>\alpha_{\mc{D}}-\delta\}$, and the continued function is holomorphic on the line $\on{Re}(s)=\alpha_\mc{D}$ except at $s=\alpha_\mc{D}$, where it has a pole, say of order $b_{\mc{D}}$.
\item Let $Z(s)=\sum_{n=1}^\infty a_n{n^{-s}}$ be a Dirichlet series such that $Z(s)=Z_\mc{D}(s-h)$ for some $h$, and assume that its abscissa of convergence $\alpha=\alpha_\mc{D}+h$ is positive. Then there exist $c,c'\in\mathbb{R}$ such that
\begin{align*}
\sum_{n=1}^N a_n\sim cN^{\alpha}(\log N)^{b_\mc{D}-1}\quad\mbox{and}\quad \sum_{n=1}^N \frac{a_n}{n^{\alpha}}\sim c'(\log N)^{b_\mc{D}}.
\end{align*} 
\item Let $K'$ be a number field including $K$, and let $\mc{D}'$ be the same collection $\mc{D}$ viewed as cone integral data over $K'$. Then $\alpha_{\mc{D}}=\alpha_{\mc{D}'}$.
\end{enumerate}
\end{theorem}
\begin{remark}
It is not true in general that $b_\mc{D}=b_{\mc{D}'}$ in (5).
In fact,  consider the cone integral data $\mc{D}=(f_0,g_0,f_1,g_1,f_2,g_2,f_3,g_3)$ over $\mb{Q}$, with polynomials in $\mb{Q}[x_{11},x_{22},x_{12}]$,  where 
$$f_0=x_{11}x_{22},\ g_0=x_{11},\ f_1=x_{11},\ g_1=x_{12},\ f_2=x_{11}x_{22},\ g_2=x_{12}^2+x_{11}^2,\ f_3=x_{11},\ g_{3}=x_{22}.$$
Let $\mc{D}'$ be the same collection $\mc{D}$ viewed as cone integral data over $\mb{Q}[i]$. An easy computation shows that
$$Z_\mc{D}(s)=\zeta(s+2)L(\chi,s+2)\quad \mbox{and}\quad Z_{\mc{D}'}(s)=(\zeta(s+2)L(\chi,s+2))^2,$$
where $\chi:\mb{N}\to\mb{C}$ is the Dirichlet character given by $\chi(n)=1$ if $n\equiv 1\mod 4$, $\chi(n)=-1$ if $n\equiv -1\mod 4$ and $\chi(n)=0$ otherwise, and $L(s,\chi)$ is the associated $L$-function. In this example we have $\alpha_\mc{D}=\alpha_{\mc{D}'}=-1$, whereas $b_\mc{D}=1$ and $b_{\mc{D}'}=2$. One can show, in the notation of Section \ref{Sec: Proof of main theorem for rings}, that $Z_\mc{D}(s-2)$ is the ideal zeta function of $\mb{Z}[i]$, whereas $Z_{\mc{D}'}(s-2)$ is the $\mb{Z}[i]$-ideal zeta function of $\mb{Z}[i]\otimes_\mb{Z}\mb{Z}[i]$.
\end{remark}
Properties (1)-(4) in Theorem \ref{basic properties of cone integrals} were proved in \cite{dSG} in the case $K=\mb{Q}$. Notice that (1) also follows from a result of Denef \cite{Den}.
The general case and (5) are somehow implicit in the arguments of \cite{dSG}. 
The rest of this section is devoted to making this more precise. Let us first state an important corollary that will be used in our study of zeta functions of groups and rings.

\begin{corollary}\label{coro:basic properties of cone integrals}
Let $\mc{D}_1$ and $\mc{D}_2$ be two cone integral data over $K$.
\begin{enumerate}
\item If $Z_{\mc{D}_1}(s,\p)=Z_{\mc{D}_2}(s,\p)$ for almost all maximal ideals $\p\subset\mc{O}_K$, then $\alpha_{\mc{D}_1}=\alpha_{\mc{D}_2}$ and $b_{\mc{D}_1}=b_{\mc{D}_2}$.
\item If there exists a number field $K'$ including $K$ such that $Z_{\mc{D}_1'}(s-h,\p')=Z_{\mc{D}_2'}(s-h,\p')$ for almost all maximal ideals $\p'\subset\mc{O}_{K'}$, where $\mc{D}_i'$ denotes the same collection $\mc{D}$ viewed as cone integral data over $K'$, then $\alpha_{\mc{D}_1}=\alpha_{\mc{D}_2}$.
\end{enumerate}
\end{corollary}
\begin{proof}
Note that (1) is an immediate consequence of Theorem \ref{basic properties of cone integrals}(2). As for (2), Theorem \ref{basic properties of cone integrals}(5) implies first that  
$\alpha_{\mc{D}_i}=\alpha_{\mc{D}_i'}$, and (1) shows that $\alpha_{\mc{D}_1'}=\alpha_{\mc{D}_2'}$. It follows that $\alpha_{\mc{D}_1}=\alpha_{\mc{D}_2}$.
\end{proof}

\medskip

From now on, we follow \cite[Sections 2, 3, 4]{dSG} with slight modifications in the notation. The interested reader may also consult \cite{Den1} for further details.
Let $K^o$ be a number field, let $\mc{O}^o=\mc{O}_{K^o}$, and let $\mc{D}^o=(f_0,g_0,\ldots,f_l,g_l)$ be a cone integral data over $K^o$.
Let $(Y^o,h^o)$ be a resolution for the polynomial $F=\prod_{i=0}^l f_i g_i\in K^o[x_1,\ldots,x_m]$ over $K^o$. Thus $Y^o$ is a closed subscheme of some projective space over $\mathbb{A}_{K^o}^m$, say $Y^o\subset\mathbb{A}_{K^o}^m\times_{K^o}\mathbb{P}_{K^o}^k$,  $h^o$ is the restriction to $Y^o$ of the projection $\mathbb{A}_{K^o}^m\times_{K^o}\mathbb{P}_{K^o}^k\to\mathbb{A}_{K^o}^m$, and the following holds:
	\begin{enumerate}[(i)]
	\item $Y^o$ is smooth over $\opn{Spec}(K^o)$;
	\item $h^o$ is an isomorphism over $\mathbb{A}_{K^o}^m\setminus V(F)$, where $V(F)\subset \mb{A}_{K^o}^m$ is the vanishing set of $F$;
	\item the reduced scheme $((h^o)^{-1}(V(F)))_{\textrm{red}}$ associated to $(h^o)^{-1}(V(F))$ has only normal crossings as subscheme of $Y^o$. 
\end{enumerate}

Let $\{E_\iota^o: \iota\in T\}$ be the set of irreducible components of $((h^o)^{-1}(V(F)))_{\text{red}}$.  These, with the structure of reduced subscheme, are smooth hypersurfaces of $Y^o$ by (iii).  For each $\iota\in T$,  let  $N_\iota(f_j)$ and $N_\iota(g_j)$ be, respectively, the multiplicities of $E_\iota^o$ in the divisor of $f_j\circ h$ and $g_j\circ h$ ($j=0,1,\ldots,l$), and let $\nu_\iota-1$ be the multiplicity of $E_\iota^o$ in the divisor of ${(h^o)}^*(dx_1\wedge\cdots \wedge dx_m)$.

We next define
\begin{align}\label{rational polyhedral cone}
	\overline{D_T}:=\left\{u\in\mathbb{R}_{\geq 0}^T:\sum_{\iota\in T} N_\iota(f_j)u(\iota)\leq \sum_{\iota\in T}N_\iota(g_j)u(\iota),\ \mbox{for}\ j=1,\ldots,l\right\},
\end{align}
and for a subset $I\subset T$ we define
\begin{align*}
D_I&:=\{u\in\overline{D_T}: u(\iota)=0\ \mbox{if and only if } \iota\in T\setminus I\},
\end{align*}
so we have $\overline{D_T}=\bigcup_{I\subset T} D_I$, a disjoint union.

Note that $\overline{D_T}$ is a rational convex polyhedral cone, so there are integral generators $\e_1,\ldots,\e_q\in\mathbb{N}_0^T\cap\overline{D_T}$ for its extremal edges such that $\mathbb{N}_0^T\cap \mathbb{R}_{\geq 0}\e_i=\mathbb{N}_0\e_i$ for $i=1,\ldots,q$. The following constants will be important: 
	\begin{align}\label{definition of A_k and B_k} 
		A_k:=\sum_{\iota\in T}\mathbf{e}_k(\iota)N_\iota(f_0)\in\mathbb{N}_0,\quad  \ B_k:=\sum_{\iota\in T} \mathbf{e}_k(\iota)(N_\iota(g_0)+\nu_\iota)\in\mathbb{N},\quad k=1,\ldots,q.
	\end{align}
	
The cone $\overline{D_T}$ has a simplicial decomposition, say $\overline{D_T}=R_0\cup R_1\cup\cdots\cup R_q\cup\cdots\cup R_w$, such that $R_0=\{0\}$, $R_i=\mathbb{R}_{>0}\e_i$ for $i=1,\ldots,q$, and $\dim R_i>1$ for $i>q$. For each $k=0,\ldots,w$ there is a subset $M_k\subset \{1,\ldots,q\}$ such that
\begin{align*}
R_k=\left\{\sum_{j\in M_k} \alpha_j\mathbf{e}_j: \alpha_j>0\ \mbox{for all } j\in M_k\right\}.
\end{align*}

Now for each $I\subset T$ there is a subset $W_I\subset \{0,\ldots,w\}$ such that $D_I=\bigcup_{k\in W_I} R_k$. 
Since $\overline{D_T}$ is the disjoint union of the $D_I$'s, $\{0,1,\ldots,w\}$ is also the disjoint union of the $W_I$'s; thus, for each $k=0,\ldots,w$, there is a unique subset  $I\subset T$ such that $k\in W_I$, and this subset will be denoted by $I_k$.

\begin{proposition}\label{formula for good primes}
Let $K$ be a number field including $K^o$, let $\mc{O}=\mc{O}_K$, and let $\mc{D}$ be the same collection $\mc{D}^o$ viewed as cone integral data over $K$. 
If $\mathfrak{p}\subset\mc{O}$ is a maximal ideal such that $(Y^o,h^o)$ has good reduction modulo $\p^o:=\p\cap\mc{O}^o\subset\mc{O}^o$, then
\begin{align}\label{eq:formula for good primes}
Z_{\mc{D}}(s,\p)=\sum_{k=0}^w(\mathbf{N}\p-1)^{|I_k|}{\mathbf{N}\p}^{-m}c_{\p,I_k}\prod_{j\in M_k}\frac{\mathbf{N}\p^{-(A_j s+B_j)}}{1-\mathbf{N}\p^{-(A_j s+B_j)}},
\end{align}
where for a subset $I\subset T$,  $c_{\p,I}:=|\{ a\in\overline{Y^o}(\mc{O}/\p) : a\in \overline{E_\iota^o}(\mc{O}/\p)\ \mbox{if and only if } \iota\in I\}|$, and for a closed subscheme
$Z^o\subset Y^o$, $\overline{Z^o}$ denotes its reduction modulo $\p^o$.
\end{proposition}
Though this proposition is more general than \cite[Corollary 3.2]{dSG}, its proof is the same, so we do not repeat it. However, we mention some comments.
Let $(Y,h)$ and $(E_\iota)_{\iota\in T}$ be the base changes of $(Y^o,h^o)$ and $(E_\iota^o)_{\iota\in T}$ by $\Spec(K_\p)\to\Spec(K^o)$.
Note that each $E_{\iota}$ is a smooth hypersurface in $Y$,  but $(E_\iota)_{\iota\in T}$ is not necessarily the collection of irreducible components of $h^{-1}(V(F))$ as
 $E_\iota$ might not be irreducible. Nevertheless, it is easy to show that 
\begin{align*}
&\on{div}(f_j\circ h)=\sum_{\iota\in T} N_\iota(f_j) E_\iota\quad\mbox{and}\quad \on{div}(g_j\circ h)=\sum_{\iota\in T} N_\iota(g_j) E_\iota\quad \mbox{for } j=0,\ldots,l,\\
&\on{div}(h^*(dx_1\wedge\cdots\wedge dx_m))=\sum_{\iota\in T}(\nu_\iota-1)E_\iota,
\end{align*}
so one can argue as in \cite[Sections 2 and 3]{dSG} by using the collection $(E_\iota)_{\iota\in T}$ and not the irreducible components of $h^{-1}(V(F))$. In fact, this seems to be what is really done in \cite{dSG}.

The following consequence was also obtained at the end of \cite[Section 2]{dSG} from a different formula for $Z_\mc{D}(s,\p)$.
\begin{corollary}\label{a_p,0 term is non-zero}
Write each $Z_\mc{D}(s,\p)$ as a power series $\sum_{i=0}^\infty a_{\p,i}\mathbf{N}\p^{-is}$. Then, $a_{\p,0}\neq 0$ for almost all maximal ideals $\p\subset\mc{O}$.
\end{corollary}
\begin{proof}
It is enough to show that $\alpha_{\p,0}\neq 0$ for almost all maximal ideals $\p\subset\mc{O}$ satysfying the hypothesis of Proposition \ref{formula for good primes}. We compute the term of the formula (\ref{eq:formula for good primes}) for $k=0$. Note that $M_0=\emptyset$ and that $I_0=\emptyset$, so the term for $k=0$ is $\mathbf{N}\p^{-m} c_{\p,\emptyset}=\mathbf{N}\p^{-m}|\{a\in \overline{Y^o}(\mc{O}/\p)\setminus \bigcup_{\iota\in T}\overline{E_\iota^o}(\mc{O}/\p)\}|=\mathbf{N}\p^{-m}|\mathbb{A}^m(\mc{O}/\p)\setminus V(\overline{F})|$, where $\overline{F}$ denotes the reduction of $F$ modulo $\p^0$ and
$V(\overline{F})\subset \mathbb{A}^m(\mc{O}/\p)$ denotes the zero set of $\overline{F}$. Note that by dimension arguments, $|\mathbb{A}^m(\mc{O}/\p)\setminus V(\overline{F})|>0$ for almost all $\p$. The proof follows as the term for $k=0$ contributes to $a_{\p,0}$. 
\end{proof} 
\medskip

Proposition \ref{formula for good primes} provides a formula for $Z_\mc{D}(s,\p)$ for almost all maximal ideals. A formula for the exceptional primes (when $K=K^o=\mb{Q}$) is given in \cite[Proposition 3.3]{dSG}, but this seems to be incorrect as it was pointed out in \cite[Remark 4.6]{AKOV2}. Nevertheless, Theorem \ref{basic properties of cone integrals} does not use this proposition but rather its consequence \cite[Corollary 3.4]{dSG}, which is correct as claimed also in \cite[Remark 4.6]{AKOV2}. We will provide a proof of this in Proposition \ref{formula for bad primes} below. We shall need the following result of Stanley (cf.\ \cite[Chapter 1]{St}), which is formulated  here in our special setting.
\begin{theorem}\label{Stanley theorem} Let $I\subset T$ be a non-empty subset, and let $\Phi_I\in\mathbb{Z}^{l\times I}$ be the matrix defined by
\begin{align*}
(\Phi_I)_{j,\iota}=N_\iota(g_j)-N_\iota(f_j),\quad j=1,\ldots,l,\ \iota\in I.
\end{align*}
Given a vector $v\in \mathbb{Z}^l$, we consider the generating series
\begin{align*}
F_{\Phi,{v}}((X_i)_{i\in I}):= \sum_{\substack{u\in \mathbb{N}_0^I:\ \Phi_I\cdot u\geq v}} X^{u},
\end{align*}
where $X^u=\prod_{\iota\in I} X_\iota^{u(\iota)}$ and $(X_\iota)_{\iota\in I}$ is a collection of variables.
If $\{u\in\mathbb{N}_0^I: \Phi\cdot u\geq v\}$ is non-empty, then $F_{\Phi,{v}}((X_i)_{i\in I})$ is a rational function whose denominator divides 
\begin{align*}
\prod_{j\in \cup_{k\in W_I} M_k} \left(1-\prod_{\iota\in  I}X^{\mathbf{e}_j(\iota)}\right).
\end{align*}
\end{theorem}
Here the inequality $\Phi_I\cdot u\geq v$ means component-wise. 
Observe that $\{\mathbf{e}_j: j\in \cup_{k\in W_I} M_k\}$ is just the set of integral generators of the extremal edges of the cone $\{u\in\mathbb{N}_0^I: \Phi_I\cdot u\geq 0\}$. These vectors can be though of as vectors in $\mathbb{N}_0^I$ since they are zero outside $I$.

\begin{proposition}\label{formula for bad primes}
Let $K$ be a number field including $K^o$, let $\mc{O}=\mc{O}_K$, and let $\mc{D}$ be the collection $\mc{D}^o$ viewed as cone integral data over $K$. If $\p\subset\mc{O}$ is any maximal ideal, then $Z_{\mc{D}}(s,\p)$ is a rational function in $\mathbf{N}\p^{-s}$ with rational coefficients, and its denominator divides 
$$\prod_{j=1}^q(1-\mathbf{N}\p^{-(A_js+B_j)}).$$
In particular, the abscissa of convergence of $Z_{\mc{D}}(s,\p)$ is either $-\infty$ or one of the rational numbers $-\frac{B_j}{A_j}$ where $j=1,\ldots,q$ and $A_j\neq 0$. 
\end{proposition}
\begin{proof}
The same notation and comments after Proposition \ref{formula for good primes} apply here. Now $h$ induces a morphism of analytic manifolds $Y({K}_{\p})\to {K}_{\p}^m$ that we still denote by $h$. This is an isomorphism above $\{\x\in {K}_{\p}^m: F(\x)\neq 0\}$. Write $Y({\mc{O}}_{\p})$ for $h^{-1}(\mc{O}_{\p}^m)$. 
Then $Y({\mc{O}}_{\p})$ can be expressed as a disjoint union of a finite number of coordinate charts, say $\{(U_b, (y_1,\ldots,y_m)): b\in\mc{B}\}$, such that for each $b\in\mc{B}$ the following holds (see \cite[Section 2]{Mo}):
\begin{enumerate}
\item The image of $(y_1,\ldots,y_m): U_b\to {K}_{\p}^m$ is exactly $(\p^{e_b}\mc{O}_{\p})^m$ for some $e_b\in\mathbb{N}_0$. 
\item Let $I_b=\{\iota\in T: E_\iota(K_\p)\cap U_b\neq \emptyset\}$. Then there exists an injective function $\mu_b:I_b\to\{1,\ldots,m\}$ such that $E_{\iota}(K_\p)\cap U_b=\{y_{\mu(\iota)}=0\}$ for all $\iota\in I_b$.
\item There are non-negative integers $c_b(f_j)$, $c_b(g_j)$ for $j=0,\ldots,l$ and $d_b$, such that the following holds on $U_b$:
\begin{align*}
|f_j\circ h|_\p&=\mathbf{N}\p^{-c_b(f_j)}\prod_{\iota\in I_b}|y_{\mu(\iota)}|_\p^{N_\iota(f_j)},\\
|g_j\circ h|_\p&=\mathbf{N}\p^{-c_b(g_j)}\prod_{\iota\in I_b}|y_{\mu(\iota)}|_\p^{N_\iota(g_j)},\\
|h^*(dx_1\wedge\cdots\wedge dx_m)|_\p&= \mathbf{N}\p^{-d_b}\prod_{\iota\in I_b}|y_{\mu(\iota)}|_\p^{\nu_\iota-1}|dy_1\wedge\cdots \wedge dy_m|_\p.
\end{align*}
\end{enumerate}
It follows that $Z_{\mc{D}}(s,\p)=\sum_{b\in \mc{B}} J_b(s)$, where
\begin{align*}
J_b(s)=\int_{V_b} \mathbf{N}\p^{-c_b(f_0)s-c_b(g_0)-d_b}\prod_{\iota\in I_b} |y_{\mu_b(\iota)}|_\p^{N_\iota(f_0)s+N_\iota(g_0)+\nu_\iota-1}|d y_1\wedge\cdots\wedge d y_m|_\p
\end{align*}
and $V_b$ is the subset of $U_b$ defined by the conditions
\begin{align*}
c_b(g_j)-c_b(f_j)+\sum_{\iota\in I_b} (N_\iota(g_j)-N_\iota(f_j))\on{ord}_\p(y_{\mu_b(\iota)})\geq 0, \quad j=1,\ldots,l.
\end{align*}
If $I_b=\emptyset$, then clearly $J_b(s)=r\mathbf{N}\p^{-c_b(f_0)s}$ for some rational number $r$.
We now analyse $J_b(s)$ for $I_b\neq\emptyset$. 
By (1) we can assume that $U_b=(\p^{e_b}\mc{O}_\p)^m$ with coordinates $y_1,\ldots,y_m$. Let $\Phi_I=\Phi_{I_b}$ be the matrix of Theorem \ref{Stanley theorem}, let  $c_b(f)-c_b(g)\in\mathbb{Z}^{l}$ denote the vector whose $j$-entry is $c_b(f_j)-c_b(g_j)$, and let $v:=c_b(f)-c_b(g)-\Phi\cdot (e_b)_{\iota\in T}\in\mb{Z}^l$.
It follows that
\begin{align*}
J_b(s)&=(1-\mathbf{N}\p^{-1})^{|I_b|}\cdot {\mathbf{N}\p^{-c_b(f_0)s-c_b(g_0)-d_b-e_b(m-|I_b|)}}\cdot \sum_{\substack{u\in(e_b\mathbb{N}_0)^{I_b}\\ \Phi\cdot u\geq c_b(f)-c_b(g)}}\prod_{\iota\in I_b} \mathbf{N}\p^{-(sN_\iota(f_0)+N_\iota(g_0)+\nu_\iota)u(\iota)}\\
&=(\mathbf{N}\p^{e_b}-\mathbf{N}\p^{e_b-1})^{|I_b|}\cdot {\mathbf{N}\p^{-c_b(f_0)s-c_b(g_0)-d_b-e_bm-\sum_{\iota\in I_b}(sN_\iota(f_0)+N_\iota(g_0)+\nu_\iota)e_b}}\\
&\quad\cdot \sum_{\substack{u\in\mathbb{N}_0^{I_0}\\ \Phi\cdot u\geq v}}\prod_{\iota\in I_b} \mathbf{N}\p^{-(sN_\iota(f_0)+N_\iota(g_0)+\nu_\iota)u(\iota)}.
\end{align*}
By Theorem \ref{Stanley theorem} this is a rational function in $\mathbf{N}\p^{-s}$ with rational coefficients, and whose denominator divides 
\begin{align*}
\prod_{j\in \bigcup_{k\in W_{I_b}} M_k} \left(1-\prod_{\iota\in I_b} \mathbf{N}\p^{-(sN_\iota(f_0)+N_\iota(g_0)+\nu_\iota)\e_j(\iota)}\right)&=\prod_{j\in \bigcup_{k\in W_{I_b}} M_k} \left(1-\prod_{\iota\in T} \mathbf{N}\p^{-(sN_\iota(f_0)+N_\iota(g_0)+\nu_\iota)\e_j(\iota)}\right)\\
&=\prod_{j\in \bigcup_{k\in W_{I_b}} M_k} \left(1-\mathbf{N}\p^{-(sA_j+B_j)}\right).
\end{align*}
This completes the proof of the proposition.
\end{proof}

With Proposition \ref{formula for good primes}, Corollary \ref{a_p,0 term is non-zero} and Proposition \ref{formula for bad primes}, the proof of properties (1)-(4) in Theorem \ref{basic properties of cone integrals} goes now exactly as in \cite[Section 4]{dSG}. In particular, one obtains the following formula for $\alpha_\mc{D}$ which is independent of the field extension $K\supset K_0$, and which proves (5) in Theorem \ref{basic properties of cone integrals}.
\begin{proposition}\label{formula for the alpha_D}
$\alpha_\mc{D}=\max\left\{\frac{1-B_j}{A_j}: j=1,\ldots, l, A_j\neq 0\right\}$. 
\end{proposition}

  \section{Zeta functions of rings as cone integrals}\label{Sec: Proof of main theorem for rings}

\noindent Let $R$ be a commutative ring with identity. By {\em an $R$-algebra} we shall mean an $R$-module $L$ endowed with an $R$-bilinear map $L\times L\to L$ called multiplication, e.g.\ an $R$-Lie algebra. 
 An $R$-subalgebra of $L$ is an $R$-submodule $A$ that is closed under multiplication. 
Left, right and two-sided $R$-ideals are defined similarly. To simplify the presentation, by an $R$-ideal we shall refer to a left $R$-ideal. Nevertheless, everything we say about left $R$-ideals is also valid for right $R$-ideals and for two-sided $R$-ideals. 
To allow applications such as Theorem \ref{main result: Conjecture} we shall also consider $R$-algebras with an identity, and in that case we will require that the $R$-subalgebras contain the identity. A $\mb{Z}$-algebra will be also called {\em a ring}.

Let $L$ be an $R$-algebra (or an $R$-algebra with identity). The $R$-subalgebra zeta function and the $R$-ideal zeta function of $L$ are by definition the formal series 
$$\zeta_L^{\s_R}(s)=\sum_{L'\leq_R L}[L:L']^{-s}\quad\mbox{and}\quad \zeta_L^{\n_R}(s)=\sum_{L'\lhd_R L}[L:L']^{-s}$$
where $L'$ runs only over those $R$-subalgebras or $R$-ideals of finite additive index respectively. We write $\zeta_L^{*_R}(s)$ to address both type of zeta functions simultaneously. We will suppress the subindex $R$ when $R=\mb{Z}$.

\medskip

Let $K$ be a number field and let $\mc{O}=\mc{O}_K$. 
   Let $L$ be an $\mc{O}$-algebra (or an $\mc{O}$-algebra with identity) whose underlying $\mc{O}$-module is free of rank $h\geq 1$.
      For a maximal ideal $\p\subset \mc O$ we write ${L}_{\p}:=L\otimes_\mc{O}{\mc O}_\p$. This is an $\mc{O}_\p$-algebra and it is easy to show that
  $$
  \zeta_{L}^{*_{\mc{O}}}(s)=\prod_{\substack{\p\subset \mc{O}\\ \textrm{maximal}}}\zeta_{{L}_\p}^{*_{\mc{O}_\p}}(s).
  $$ 
We now fix a basis for $L$ as $\mc{O}$-module and hence identify $L$ with $\mc{O}^h$. Let 
$(c_{ij}^k)$ be the structure coefficients of $L$ with respect to the canonical basis $\{e_1,\ldots,e_h\}$, that is, $e_i\cdot e_j=\sum_{k=1}^h c_{ij}^k e_k$.
We denote by $\on{Tr}_h(\mc{O}_\p)$ the set of upper-triangular matrices with entries in $\mc{O}_\p$ and write $|dm|_\p$ for the normalized additive Haar measure of $\on{Tr}_h(\mc{O}_\p)$.
Given a matrix $m=(m_{ij})$ we denote by $m'=(m_{ij}')$ its adjoint.  
\begin{proposition}\label{expression of zeta functions as cone integrals}
For each maximal ideal $\p\subset\mc{O}$ it holds that	
  \begin{align*}
	\zeta_{{L}_\p}^{*_{\mc{O}_\p}}(s)=(1-\mathbf{N}\p^{-h})^{-1}\int_{\mc{M}^*(\mc{O}_\p)} |m_{11}|_\p^{s-1}\cdots|m_{hh}|_\p^{s-h}|dm|_\p, 
\end{align*}
where $\mc{M}^\s(\mc{O}_\p)\subset\on{Tr}_h({\mc O}_\p)$ denotes the set of those upper-triangular matrices $(m_{ij})$ such that 
\begin{align*}
 \on{ord}_\p(m_{11}\cdots m_{hh})&\leq \on{ord}_\p\left(\sum_{t=1}^h\sum_{r=i}^h\sum_{s=j}^h m_{ir}c_{rs}^t m_{js} m'_{tk}\right)\quad \mbox{for all } i,j,k=1,\ldots,h,
\end{align*}
and $\mc{M}^\n(\mc{O}_\p)\subset \on{Tr}_h({\mc O}_\p)$ denotes the set of those upper-triangular matrices $(m_{ij})$  such that 
\begin{align*}
	\on{ord}_\p(m_{11}\cdots m_{hh})&\leq \on{ord}_\p\left(\sum_{s=1}^h\sum_{r=i}^h m_{ir}c_{rj}^s m'_{sk}\right)\quad \mbox{for all }i,j,k=1,\ldots,h.
	\end{align*}

If in addition $L$ is an $\mc{O}$-algebra with identity, say $1=(u_1,\ldots,u_h)$, then we also have to add the following extra conditions in the definition of $\mc{M}^\s(\mc{O}_\p)$:
\begin{align*}
\on{ord}_\p(m_{11}\cdots m_{hh})\leq \on{ord}_\p\left(\sum_{i=1}^h u_im_{ij}'\right)\quad \forall j=1,\ldots,h.
\end{align*}
\end{proposition}
\begin{proof} 
The first part of the proposition when $K=\mb{Q}$ is proved in \cite[Theorem 5.5]{dSG} (see also \cite[Sec.\ 3]{GSS}). That proof can be easily extended to this general case. By following that proof in the case that $L$ is an $\mc{O}$-algebra with identity, we see that the conditions that we add to the definition of $\mc{M}^\s(\mc{O}_\p)$  are simply a translation of the condition that the $\mc{O}_\p$-submodule of $L_\p$ generated by the rows of the matrix $(m_{ij})$ contains $1
=(u_1,\ldots,u_h)$. This is necessary since we are requiring that the $R$-subalgebras of an $R$-algebra with identity must contain 1.
\end{proof} 
As an immediate consequence we obtain
\begin{corollary}\label{existence of the cone integral data}
Let $L$ be an $\mc{O}$-algebra (or an $\mc{O}$-algebra with identity) that is isomorphic to $\mc{O}^h$ as $\mc{O}$-module.
Then there exists a cone integral data $\mc{D}^{*}$ over $K$ such that
$$\zeta_{L}^{*_\mc{O}}(s)=Z_{\mc{D}^{*}}(s-h),$$
and such that for any finite extension $K'\supset K$, say with ring of integers $\mc{O}'$, we have
$$\zeta_{L'}^{*_{\mc{O}'}}(s)=Z_{\mc{D'}^*}(s-h),$$
where $L'$ is the $\mc{O}'$-algebra $L\otimes_\mc{O}\mc{O}'$ and $\mc{D'}^*$ is the same collection $\mc{D}^*$ viewed as cone integral data over $K'$. 
\end{corollary}
As a combination of this corollary and Theorem \ref{basic properties of cone integrals} we now obtain
\begin{corollary}\label{invariance of a_L under base extension} 
Let $L$ be an $\mc{O}$-algebra (or an $\mc{O}$-algebra with identity) that is isomorphic to $\mc{O}^h$ as $\mc{O}$-module. Assume that $\zeta_L^{*_\mc{O}}(s)$ is not a constant function, and let $\alpha_L^{*_{\mc{O}}}$ be its abscissa of convergence. Then the following holds.
\begin{enumerate} 
\item $\alpha_L^{*_{\mc{O}}}$ is a rational number and there exists $\delta>0$ such that $\zeta_L^{*_\mc{O}}(s)$ can be meromorphically continued to the region $\on{Re}(s)>\alpha_L^{*_{\mc{O}}}-\delta$. Moreover, the continued function is holomorphic on the line $\on{Re}(s)=\alpha_L^{*_{\mc{O}}}$ except at $s=\alpha_L^{*_{\mc{O}}}$ where it has a pole.
\item Let $K'$ be a number field including $K$, $\mc{O}'$ its ring of integers, and $L'=L\otimes_\mc{O}\mc{O}'$. 
Then $\zeta_{L}^{*_\mc{O}}(s)$ and $\zeta_{L'}^{*_{\mc{O}'}}(s)$ have the same abscissa of convergence.
\end{enumerate}
\end{corollary}
Another consequence is
\begin{corollary}\label{invariance of a_L and b_L by commmensurability}
Let $L_1$ and $L_2$ be two $\mc{O}$-algebras (or $\mc{O}$-algebras with identity) that are isomorphic to $\mc{O}^h$ as $\mc{O}$-modules. Let $b_{L_i}^{*_\mc{O}}$ be the order of the pole of $\zeta_{L_i}^{*_\mc{O}}(s)$ at $s=\alpha_{L_i}^{*_\mc{O}}$. If $L_1\otimes_\mc{O}K$ and $L_2\otimes_\mc{O}K$ are isomorphic $K$-algebras, then $\alpha_{L_1}^{*_\mc{O}}=\alpha_{L_2}^{*_\mc{O}}$ and $b_{L_1}^{*_\mc{O}}=b_{L_2}^{*_\mc{O}}$.
\end{corollary}
\begin{proof}
Let $\mc{D}_i^*$ be the cone integral data of Corollary \ref{existence of the cone integral data} for $L_i$. By Corollary \ref{basic properties of cone integrals}(1) it is enough to prove that $Z_{\mc{D}_1^*}(s-h,\p)=Z_{\mc{D}_2^*}(s-h,\p)$ for almost all maximal ideals $\p\subset\mc{O}$. This follows from the fact that ${L_1}_\p:=L_1\otimes_\mc{O}\mc{O}_\p$ and ${L_2}_\p:=L_2\otimes_\mc{O}\mc{O}_\p$ are isomorphic $\mc{O}_\p$-algebras for almost all $\p$. 
To prove this fact, we may assume that $L_1=L_2=\mc{O}^h$ as $\mc{O}$-modules and so $L_1\otimes_\mc{O}K=L_2\otimes_\mc{O}K=K^h$. Let $\alpha\in \on{GL}_h(K)$ be a $K$-algebra isomorphism $\alpha:L_1\otimes_\mc{O}K\to L_2\otimes_\mc{O}K$. Then for almost all maximal ideals $\p$ it holds that $\alpha,\alpha^{-1}\in\on{GL}_h(\mc{O}_\p)$, and hence $\alpha$ induces an isomorphism ${L_1}_\p\cong {L_2}_\p$.  This completes the proof.
\end{proof}

\medskip

In order to prove Theorem \ref{main theorem: general version} we shall need the following result.
\begin{lemma}\label{an algebraic geometry lemma}
Let $K$ be a field and let $A_1$ and $A_2$ be finite dimensional $K$-algebras.
If for some field extension ${K}'\supset K$ the $K'$-algebras $A_1\otimes_K{K}'$ and $A_2\otimes_K {K}'$ are isomorphic over $K'$, then this also holds for some finite extension ${K}'\supset K$.
\end{lemma}
\begin{proof}
Note that $A_1$ and $A_2$ have necessarily the same dimension over $K$, hence we can assume that $A_1=A_2=K^h$ as vector spaces over $K$ for some $h$.
Let $I\subset \on{GL}_h\times_\mb{Z}\Spec(K)$ be the subfunctor such that if ${K}'$ is any commutative algebra over $K$ with identity, then $I({K}')$ is the set of ${K}'$-algebra isomorphisms $A_1\otimes_K {K}'\to A_2\otimes_K {K}'$.
It is easy to see that $I$ is represented by a closed subscheme of $\on{GL}_h\times_\mb{Z}\Spec(K)$.
The hypothesis implies that $I$ is not the empty scheme. Therefore, if ${K}'$ is the residue field at a closed point of $I$, which is a finite extension of $K$, we have $I({K}')\neq\emptyset$.
\end{proof}
Theorem \ref{main theorem: general version} follows from the following
\begin{theorem}
Let $K$ be a number field and $\mc{O}$ its ring of integers. Let $L_1$ and $L_2$ be two $\mc{O}$-algebras (or $\mc{O}$-algebras with identity) that are isomorphic to $\mc{O}^h$ as $\mc{O}$-modules for some $h>0$. If $L_1\otimes_\mc{O} {K}'$ and $L_2\otimes_\mc{O}{K}'$ are isomorphic ${K}'$-algebras for some field extension ${K}'\supset K$, then  $\zeta_{L_1}^{*_\mc{O}}(s)$ and $\zeta_{L_2}^{*_\mc{O}}(s)$ have the same abscissa of convergence.
\end{theorem}
\begin{proof}
By Lemma \ref{an algebraic geometry lemma} we can assume that the field ${K}'$ of the hypothesis is a finite extension of $K$. 
Let ${\mc{O}}'$ be its ring of integers and let $L_i'=L_i\otimes_\mc{O}\mc{O}'$.
By Corollary \ref{invariance of a_L under base extension}, $\zeta_{L_i}^{*_\mc{O}}(s)$ and $\zeta_{L_i'}^{*_{\mc{O}'}}(s)$ have the same abscissa of convergence for $i=1,2$. 
By Corollary \ref{invariance of a_L and b_L by commmensurability}, $\zeta_{L_1'}^{*_{\mc{O}'}}(s)$ and $\zeta_{L_2'}^{*_{\mc{O}'}}(s)$ have the same abscissa of convergence since $L_1'\otimes_{\mc{O}'}K'=L_2'\otimes_{\mc{O}'}K'$ as $K'$-algebras. It follows that $\zeta_{L_1}^{*_\mc{O}}(s)$ and $\zeta_{L_2}^{*_\mc{O}}(s)$ have the same abscissa of convergence. 
\end{proof}

We now explain how Theorem \ref{main theorem} follows from Theorem \ref{main theorem: general version}.
Let $\mf{N}$ be a unipotent group scheme over $\mb{Q}$ and let $N$ be an arithmetic group of $\mf{N}$. We define $\mf{n}$ to be the Lie algebra of $\mf{N}$, which is a nilpotent Lie algebra of dimension $h(N)$ over $\mb{Q}$. 
\begin{proposition}\label{the Mal'cev correspondence}
Let $L$ be any Lie subring of $\mf{n}$ additively isomorphic to $\mb{Z}^h$, with $h=h(N)$, such that $L\otimes_\mb{Z}\mb{Q}=\mf{n}$. Then 
$\zeta_N^*(s)$ and $\zeta_L^*(s)$ have the same abscissa of convergence, that is $\alpha_N^*=\alpha_L^*$. In addition, $b_N^*=b_L^*$.
\end{proposition} 
\begin{proof}
By Corollary \ref{invariance of a_L and b_L by commmensurability} it is enough to prove this by just one $L$. By \cite[Theorem 4.1]{GSS}, there is $L$ satisfying the hypothesis such that for almost all primes $p$, $\zeta_N^*(s)$ and $\zeta_L^*(s)$ have the same local factor at $p$.
Now, by Corollary \ref{existence of the cone integral data}, there exists a cone integral data $\mc{D}_L^*$ such that $\zeta_L^*(s)=Z_{\mc{D}_L^*}(s-h)$. By \cite[Corollary 1]{Su}, there exists a cone integral data $\mc{D}_N^*$ such that $\zeta_N^*(s)=Z_{\mc{D}_N^*}(s-h)$. It follows that $Z_{\mc{D}_N^*}(s-h,p)=Z_{\mc{D}_L^*}(s-h,p)$ for almost all $p$. The proposition follows now from Corollary \ref{coro:basic properties of cone integrals}(1).
\end{proof}
Since the category of unipotent algebraic groups over a field $K$ of characteristic zero and the category of finite dimensional nilpotent Lie algebras over $K$ are equivalent, it is now clear that Theorem \ref{main theorem} follows from Theorem \ref{main theorem: general version}.

\section{An upper bound for $\alpha_N^*$ for non-abelian $\mf{T}$-groups}\label{Sec: Proof of Theorem A}
\noindent In this section we prove Theorem \ref{main theorem:upper bound for a(N)}. By Proposition \ref{the Mal'cev correspondence}, it is enough to prove the analogous result for nilpotent Lie rings. 
\begin{lemma}\label{lemma with a bound}
Let $\mb{Z}^e$ be the free abelian group of rank $e$, $e>0$, and let $\delta$ be a positive number such that $\delta<e$. Then there is a constant $k=k(e,\delta)$ such that
\begin{align*}
\sum_{n=1}^N\frac{a_n^\s(\mb{Z}^e)}{n^{e-\delta}}\leq k N^\delta,\quad\forall N\in\mb{N}.
\end{align*}
\end{lemma}
\begin{proof}
Consider the Dirichlet series $\ds Z(s)=\sum_{n=1}^\infty\frac{b_n}{n^s}$, where $\ds b_n=\frac{a_n^\s(\mb{Z}^e)}{n^{e-\delta}}$.
Then $Z(s)=\zeta_{\mb{Z}^e}^\s(s+e-\delta)=Z_\mc{D}(s-\delta)$, where $\mc{D}$ is the cone integral data of Corollary \ref{existence of the cone integral data} for the zeta function $\zeta_{\mb{Z}^e}^\s(s)$ (we view $\mb{Z}^e$ as an abelian Lie ring). The abscissa of convergence of $Z(s)$ is $\delta>0$, and the order of the pole of $Z(s)$ at $s=\delta$, which is the order of the pole of $\zeta_{\mb{Z}^e}^\s(s)$ at $e$, is 1. This follows from the formula for $\zeta_{\mb{Z}^e}^\s(s)$ given in Example \ref{zeta function of Z^d}. 
Thus, by Theorem \ref{basic properties of cone integrals}(4), there exists a constant $k$ such that $\sum_{n=1}^N b_n\leq k N^\delta$ for all $N$. This proves the lemma.
\end{proof}

We now set some notation. If $L$ is a ring additively isomorphic to $\mb{Z}^h$, then $\alpha_L^\lhd\leq\alpha_L^\s\leq\alpha_{\mb{Z}^h}^\s=h$. We denote $\delta_L^*=h-\alpha_L^*$ for $*\in\{\leq,\lhd\}$.
If $L$ is a nilpotent Lie ring, then $\gamma_i(L)$ denotes the $i$-th term of the lower central series. The nilpotency class of $L$ is the first positive integer $c$ such  that $\gamma_{c+1}(L)=0$. If $A$ is an ideal of $L$ we also define a lower series as follows: $\gamma_1(L,A)=A$ and $\gamma_i(L,A)=[\gamma_{i-1}(L,A),L]$ for $i>1$.
\begin{proposition}\label{prop: bounds for the abscissa}
Let $L$ be a non-abelian nilpotent Lie ring additively isomorphic to $\mb{Z}^h$, let $c$ be its nilpotency class, let $Z=\{x\in L: nx\in \gamma_c(L)\ \mbox{for some } n\in\mathbb{N}\}$, and let $e$ be the additive rank of $\gamma_c(L)$.
Then  $$\delta_L^\s\geq \frac{\delta_{L/Z}^\s+e}{1+(c-1)e}\quad\mbox{and}\quad \delta_L^\n\geq \frac{\delta_{L/Z}^\n+e}{1+e}.$$
\end{proposition}
\begin{proof} 
Notice that $Z$ is an ideal of $L$ included in the centre of $L$, the quotient ring $L/Z$ is additively isomorphic to $\mb{Z}^{h-e}$, and the index $k_1:=[Z:\gamma_c(L)]$ is finite. Since $Z$ is central, any subgroup of $Z$ is automatically an ideal of $L$.

If $A$ is a finite index subring of $L$, then $A+Z$ is a subring including $Z$, and $A\cap Z$ is a finite index subgroup of $Z$. If $A$ is an ideal, then $A+Z$ is also an ideal. In any case, we have $[L:A]=[L:A+Z][Z:A\cap Z]$.  It is easy to show that if $A$ is a subring, then $\gamma_c(A+Z)\subset A\cap Z$, and if $A$ is an ideal, then $\gamma_c(L, A+Z)\subset A\cap Z$. Thus, for any positive $\delta$ with $\delta<h$ we have
\begin{align*}
\zeta_{L}^\s(h-\delta)&=\sum_{A\leq_f L}[L:A]^{-h+\delta}=\sum_{\substack{B\leq_f L:\\ Z\subset B}}[L:B]^{-h+\delta}\sum_{\substack{C\leq_f Z:\\ \gamma_c(B)\subset C}}[Z:C]^{-h+\delta}\mu_{B,C}^\s,\\
\zeta_{L}^\n(h-\delta)&=\sum_{A\lhd_f L}[L:A]^{-h+\delta}=\sum_{\substack{B\lhd_f L:\\ Z\subset B}}[L:B]^{-h+\delta}\sum_{\substack{C\leq_f Z:\\ \gamma_c(L,B)\subset C}}[Z:C]^{-h+\delta}\mu_{B,C}^\n,
\end{align*}
where $\mu_{B,C}^\s=|\{A\leq_f L: A+Z=B, A\cap Z=C\}|$ and $\mu_{B,C}^\n=|\{A\lhd_f L: A+Z=B, A\cap Z=C\}|$.
By \cite[Lemma 1.3.1]{LS} we have $\mu_{B,C}^\n\leq\mu_{B,C}^\s\leq |\on{Hom}(B/Z,Z/C)|\leq [Z:C]^{h-e}$. Thus,
\begin{align}\label{eq11}
\zeta_L^\s(h-\delta)&\leq \sum_{\substack{B\leq_f L:\\ Z\subset B}}[L:B]^{-h+\delta}\sum_{\substack{C\leq_f Z:\\ \gamma_c(B)\subset C}}[Z:C]^{-e+\delta},\\
\nonumber \zeta_L^\n(h-\delta)&\leq \sum_{\substack{B\lhd_f L:\\ Z\subset B}}[L:B]^{-h+\delta}\sum_{\substack{C\leq_f Z:\\ \gamma_c(L,B)\subset C}}[Z:C]^{-e+\delta}.  
\end{align}

Given $B\leq_f L$ with $Z\subset B$, we claim that $\gamma_c(B)\supset [L:B]^{c-1}\gamma_c(L)$. In fact, there are $x_1,\ldots,x_{h-e}\in Z$ whose classes modulo $Z$ form a basis of the $\mb{Z}$-module $L/Z$, and there are positive integers $d_1|d_2|\cdots |d_{h-e}$ such that $B/Z$ is generated by the classes of $d_1x_1,\ldots,d_{h-e}x_{h-e}$. So $[L:B]=d_1d_2\cdots d_{h-e}$. Notice that $\gamma_c(L)$ is the subgroup of $Z$ generated by all the elements of the form $[x_{i_1},\ldots,x_{i_c}]:=[[\cdots[[x_{i_1},x_{i_2}],x_{i_3}],\cdots],x_{i_{c}}]$ with not all $i_1,\ldots,i_c$ equal, and $\gamma_c(B)$ is generated by all the elements of the form $d_{i_1}\cdots d_{i_c}[x_{i_1},\ldots,x_{i_c}]$ with not all $i_1,\ldots,i_c$ equal. Now, if $i_1,\ldots,i_c$ are not all equal to each other, then the product of two factors in $d_{i_1}\cdots d_{i_c}$ divides $[L:B]$ and the other factors are also divisors of $[L:B]$. Thus, $d_{i_1}\cdots d_{i_c}$ divides $[L:B]^{c-1}$ and hence $[L:B]^{c-1}[x_{i_1},\ldots,x_{i_c}]\in  \gamma_c(B)$. It follows that $\gamma_c(B)\supset [L:B]^{c-1}\gamma_c(L)$, as claimed.   Note that if $B$ is in addition an ideal, then $\gamma_c(L,B)\supset \gamma_c(L,[L:B]L)= [L:B]\gamma_c(L)$. 
Since 
$[L:[L:B]^t\gamma_c(L)]=k_1[\gamma_c(L): [L:B]^t\gamma_c(L)]=k_1[L:B]^{te}$ for any $t\in\mb{N}$, we conclude that 
\begin{align}\label{eq22}
\sum_{\substack{C\leq_f Z:\\ \gamma_c(B)\subset C}}[Z:C]^{-e+\delta} &\leq \sum_{\substack{C\leq_f Z\\ [Z:C]\leq k_1[L:B]^{(c-1)e}}}[Z:C]^{-e+\delta}\leq k_2 (k_1 [L:B]^{(c-1)e})^{\delta},\\
\nonumber \sum_{\substack{C\leq_f Z:\\ \gamma_c(L,B)\subset C}}[Z:C]^{-e+\delta} &\leq \sum_{\substack{C\leq_f Z\\ [Z:C]\leq k_1[L:B]^{e}}}[Z:C]^{-e+\delta}\leq k_2 (k_1 [L:B]^{e})^{\delta},
\end{align}
where $k_2$ is the constant provided by Lemma \ref{lemma with a bound}. 

A combination of (\ref{eq11}) and (\ref{eq22}) yields
\begin{align*}
\zeta_L^\s(h-\delta)&\leq k(\delta)\sum_{Z\subset B\leq_f L} [L:B]^{-h+\delta+\delta (c-1)e}=k(\delta)\zeta_{L/Z}^\s(h-\delta(1+(c-1)e)),\\
\zeta_L^\n(h-\delta)&\leq k(\delta)\sum_{Z\subset B\lhd_f L} [L:B]^{-h+\delta+\delta e}=k(\delta)\zeta_{L/Z}^\n(h-\delta(1+e)),
\end{align*}
for some constant $k(\delta)$.
It follows from the first inequality that $\zeta_L^\s(h-\delta)$ converges if $h-\delta(1+ce)>\alpha_{L/Z}^\s$, i.e., if $\delta<\frac{h-\alpha_{L/Z}^\s}{1+ce}=\frac{\delta_{L/Z}^\s+e}{1+(c-1)e}$. This proves that $\delta_{L/Z}^\s$ is at least $\frac{\delta_{L/Z}^\s+e}{1+(c-1)e}$. Similarly, from the second inequality we deduce that $\delta_L^\n$ is at least $\frac{\delta_{L/Z}^\n+e}{1+e}$.
\end{proof}

Theorem \ref{main theorem:upper bound for a(N)} follows from the next one.
\begin{theorem}\label{bounds for a_N}
Let $L$ be a non-abelian nilpotent Lie ring additively isomorphic to $\mb{Z}^h$, and let $c$ be its nilpotency class. 
\begin{enumerate}
\item If $c=2$, then $\alpha_L^\s\leq h-\frac{1}{2}$. If $c>2$ then $\alpha_L^\s\leq h-\frac{1}{c-1}$.
\item $\alpha_L^\n\leq h-1$.
\end{enumerate}
\end{theorem}
\begin{proof}
 We use the notation of Proposition \ref{prop: bounds for the abscissa}. Observe that the nilpotency class of $L/Z$ is $c-1$.

We prove (1). Assume first that $c=2$. 
By Proposition \ref{prop: bounds for the abscissa}, $\delta_L^\s\geq \frac{e}{1+e}\geq \frac{1}{2}$, hence $\alpha_L^\s\leq h-\frac{1}{2}$. 
We now assume that $c>2$ and prove that $\alpha_L^\s\leq h-\frac{1}{c-1}$, or equivalently that $\delta_L^\s\geq \frac{1}{c-1}$, by induction on $c$. 
If $c=3$, then $\delta_L^\s\geq\frac{\delta_{L/Z}^\s+e}{1+2e}\geq\frac{1/2+e}{1+2e}=\frac{1}{2}=\frac{1}{c-1}$. Assume next that $c>3$ and that the result has been proved for $c-1$ (in particular for $L/Z$). Then $\delta_L^\s\geq \frac{\delta_{L/Z}^\s+e}{1+(c-1)e}\geq\frac{1/(c-2)+e}{1+(c-1)e}>\frac{1/(c-1)+e}{1+(c-1)e}=\frac{1}{c-1}$. This completes the induction and the proof of (1).

We now prove (2), which is equivalent to $\delta_L^\n\geq 1$, by induction on $c$. If $c=2$, then the result follows from \cite[Proposition 6.3]{GSS}. 
  Assume now that $c>2$ and that the result has been proved for $c-1$. By Proposition \ref{prop: bounds for the abscissa} we have $\delta_L^\n\geq \frac{\delta_{L/Z}^\n+e}{1+e}\geq \frac{1+e}{1+e}=1$. This completes the induction and the proof of (2).
\end{proof}

\section{A version of Theorem \ref{main theorem} for virtually nilpotent groups}\label{Sec: Main theorem for virtually nilpotent groups}

Let $G$ be a finitely generated virtually nilpotent group and let $N$ be its Fitting subgroup, that is, the maximal nilpotent normal subgroup. It is known that $\alpha_G^*\leq \alpha_N^*+1$ (cf. \cite[Proposition 5.6.4]{LS}, \cite[Theorem 3]{Su}), and the next example shows that the equality might hold. 
\begin{example}\label{example of virtually abelian group}
Let $N=\mb{Z}$ and $G=\mb{Z}\rtimes \on{Aut}(\mb{Z})$. Then $\zeta_N^\s(s)=\zeta(s)$ and $\zeta_G^\s(s)=2^{-s}\zeta(s)+\zeta(s-1)$. In particular, $\alpha_N^\s=1$ and $\alpha_G^\s=2$.
\end{example} 
It follows that $\alpha_G^*$ is not longer  commensurability-invariant within the class of finitely generated virtually nilpotent groups. 
This notwithstanding, it is possible to formulate a version of Theorem \ref{main theorem} for virtually nilpotent groups. First of all, given $G$ and $N$ as above, we may assume that $N$ is torsion-free, that is, a $\mf{T}$-group. In fact, the torsion subgroup $t(N)$ of $N$ is a finite normal subgroup of $G$ and the next lemma shows that $\alpha_G^*=\alpha_{G/t(N)}^*$.
\begin{lemma}\label{LS Proposition 5.6.4}
Let $G$ be a group of finite rank and $T$ a finite normal subgroup of $G$. Write $Q=G/T$.
Then $\alpha_G^*=\alpha_Q^*$ and $\zeta_{G,p}^*(s)=\zeta_{Q,p}^*(s)$ for every prime $p$ not dividing $|T|$.
\end{lemma}
\begin{proof}
The equality $\alpha_G^\s=\alpha_Q^\s$ is proved in \cite[Proposition 5.6.2]{LS}. 
We adapt that proof to show that $\alpha_G^\n=\alpha_Q^\n$. Clearly $\alpha_G^\n\geq \alpha_Q^\n$, so we only need to focus on the reverse inequality.

Fix a positive integer $n$. A normal subgroup $H\lhd G$ of index $n$ determines normal subgroups $H\cap T\lhd T$ and $HT\lhd G$. The index $[T: H\cap T]$ divides both $n$ and $|T|$, and we have $[G:HT]=n/[T:H\cap T]$. 
Now fix a common divisor, say $t$, of $|T|$ and $n$. Fix also $D\lhd T$ and $B\lhd G$  such that $T\subset B$, $[T:D]=t$ and $[G:B]=n/t$.
If there is $H\lhd G$ such that $H\cap T=D$ and $HT=B$, then necessarily $[G:H]=n$, $D$ is normal in $G$, and $H/D$ is a complement of $T/D$ in $HT/D$. 
Therefore, there are at most $|\on{Der}(B/T,T/D)|$ possibilities for $H$, and this number is turn bounded by $[T:D]^{\on{rk}(Q)}\leq |T|^{\on{rk}(Q)}$; cf.\ \cite[Lemma 1.3.1]{LS}. It follows that
\begin{align*}
a_n^\n(G)\leq \sum_{t|n,\ t||T|} a_t^\n(T)a_{n/t}^{\n}(Q)|T|^{\on{rk}(Q)}\leq |T|^{\on{rk}(Q)} \zeta_T^\n(0)\sum_{t|n,\ t||T|} a_{n/t}^\n(Q).
\end{align*}

By using the above inequality we now conclude that for any positive integer $n$,
\begin{align*}
\sum_{j=1}^n a_j^\n(G)\leq |T|^{\on{rk}(Q)}\zeta_T^\n(0) \sum_{t| |T|}\sum_{t|j\leq n} a_{j/t}^\n(Q)\leq |T|^{\on{rk}(Q)+1}\zeta_T^\n(0)\sum_{j=1}^n a_j^\n(Q)
\end{align*}
and this clearly implies that $\alpha_G^\n\leq \alpha_Q^\n$.

The equality $\zeta_{G,p}^*(s)=\zeta_{Q,p}^*(s)$ for a prime $p\nmid |T|$ holds because the index of any subgroup $H\leq G$ is divisible by $[TH:H]=[T:T\cap H]$,  hence if $[G:H]$ is a power of $p$ we have $[T:T\cap H]=1$, that is $H\supseteq T$.
\end{proof}

We return to the set-up introduced at the beginning of the section. We will also assume from now on that $N$ is a $\mf{T}$-group. Consider now the induced group extension $S:1\to N\to G\xrightarrow{\pi} F\to 1$. It is easy to check that 
\begin{align*}
	\zeta_G^\s(s)=\sum_{E\leq F}[F:E]^{-s}\zeta_{S,E}^{\s}(s),\quad \quad  \zeta_G^\lhd(s)=\sum_{E\lhd F}[F:E]^{-s}\zeta_{S,E}^{\n}(s),
\end{align*}
where 
$$\ds	\zeta_{S,E}^{\s}(s):=\sum_{\substack{A\leq_f G:\ \pi(AN)=E}}[\pi^{-1}(E):A]^{-s},\quad \quad \ds\zeta_{S,E}^{\n}(s):=\sum_{\substack{A\lhd_f G:\ \pi(AN)=E}}[\pi^{-1}(E):A]^{-s}.$$
We denote by $\alpha_{S,E}^\s$ or $\alpha_{S,E}^\n$ the abscissae of convergence of $\zeta_{S,E}^\s(s)$ and $\zeta_{S,E}^\n(s)$. We recall the following result.
\begin{theorem}[{\cite{Su}}]\label{main resulta of Su}
Let $*\in\{\leq,\lhd\}$, and let $E\leq F$, where $E$ is normal if $*=\lhd$. 
Then there exists a cone integral data $\mc{D}_E^*$ over $\mb{Q}$ such that $\zeta_{S,E}^{*}(s)=Z_{\mc{D}_E*}(s-h(N)-|E|+1)$. Therefore, $\alpha_{S,E}^*$ is a rational number and $\zeta_{S,E}^*(s)$ has meromorphic continuation to a region of the form $\on{Re}(s)>\alpha_{S,E}^*-\delta$ for some $\delta>0$.  
\end{theorem}
It follows that  $\alpha_G^*$ is a rational number and that $\zeta_G^*(s)$ has meromorphic continuation to a region of the form $\on{Re}(s)>\alpha_G^*-\delta$ for some $\delta>0$.

\subsection{The case of virtually abelian groups}
To motivate the formulation of Theorem \ref{main theorem} for virtually nilpotent groups, we will make a digression and discuss the case where $N$ is abelian. 
We will change the notation and write $T$ instead of $N$.
In this case, a formula for $\zeta_{S,E}^{*}(s)$ (up to a finite number of local factors) was given in \cite{dSMS}, and this suffices to read off the abscissa of convergence (Theorem \ref{basic properties of cone integrals}(2)). We will recall this  result (see Proposition \ref{main result of dSMS} below) after introducing some notation.

Let $F$ be a finite group and let $V$ be a $\mb{Q}[F]$-module of finite dimension over $\mb{Q}$. 
\begin{enumerate}
	\item  Let $\mb{Q}[F]=A_0\oplus A_1\oplus\cdots A_r$ be a decomposition of $\mb{Q}[F]$ into simple components. Then $A_i$ is isomorphic to a matrix algebra $M_{m_i}(D_i)$ for some central division algebra $D_i$ over a number field $K_i$. Assume that  $A_0=\mb{Q}\cdot \sum_{\gamma\in F}\gamma$, so that $D_0=K_0=\mb{Q}$ and $m_0=1$. 
	\item  Let $n_i^2=\dim_{K_i} A_i=m_i^2e_i^2$, where $e_i^2=\dim_{K_i} D_i$.
	\item Let $V=V_0\oplus V_1\oplus\cdots \oplus V_r$, where $V_i=A_iV$. Then $V_i\cong (D_i^{m_i})^{k_i}$ (as $A_i$-modules) for some integer $k_i\geq 0$. Note that $V_0\cong \mb{Q}^{k_0}$ is the set of fixed points of $F$.
\end{enumerate}
 Define 
\begin{align*}
	\zeta_{F\curvearrowright V}^\s(s):=\prod_{j=0}^{k_0-1}\zeta(s-j)\cdot \prod_{i=1}^{r}\prod_{j=0}^{k_ie_i-1}\zeta_{K_i}(n_i(s-1)-j),\quad 	\zeta_{F\curvearrowright V}^\n(s):=\prod_{j=0}^{k_0-1}\zeta(s-j)\cdot \prod_{i=1}^{r}\prod_{j=0}^{k_ie_i-1}\zeta_{K_i}(n_is-j), 
\end{align*}
where $\zeta_{K_i}(s)$ is the Dedekind zeta function of $K_i$.
Let $\alpha_{F\curvearrowright V}^*$ denote the abscissa of convergence of $\zeta_{F\curvearrowright V}^*(s)$. Since $\zeta_{K_i}(s)$ has abscissa of convergence 1, we conclude that 
\begin{align*}
	\alpha_{F \curvearrowright V}^\s=\max\left\{ k_0,\frac{k_1}{m_1}+1,\ldots,\frac{k_r}{m_r}+1\right\},\quad \alpha_{F\curvearrowright V}^\n=\max\left\{ k_0,\frac{k_1}{m_1},\ldots,\frac{k_r}{m_r}\right\}.
\end{align*} 
  
\begin{remark}\label{remark on the abscissa for virtually abelian groups=h+1}
A quick analysis shows that either $\alpha_{F\curvearrowright V}^{\s}\leq \dim V$ or else $\alpha_{F\curvearrowright V}^{\s}=\dim V+1$, in which case the action of $F$ on $V$ is non-trivial, every $f\in F$ acts on $V$ either as the identity or as minus the identity, and $\zeta_{F\curvearrowright V}^{\s}(s)=\zeta_V(s-1)$, where $\zeta_V(s):=\prod_{j=0}^{\dim V-1}\zeta(s-j)$.
\end{remark}
\begin{remark}\label{remark on the abscissa for virtually abelian groups and base change} 
Note that after base change with $\mb{C}$, $A_i$ decomposes as a product of $[K_i:\mb{Q}]$ copies of $M_{m_ie_i}(\mb{C})$, and $V_i$ decomposes accordingly as a product of $[K_i:\mb{Q}]$ copies of $(\mb{C}^{m_ie_i})^{e_ik_i}$. Since $\frac{e_ik_i}{e_im_i}=\frac{k_i}{m_i}$, the numbers $\alpha_{F\curvearrowright V}^\s$ and $\alpha_{F\curvearrowright V}^\n$   can still be read off from the $\mb{C}[F]$-module $V\otimes_\mb{Q}\mb{C}$. Hence:
\end{remark} 
\begin{corollary}\label{coro: invariance of a(F,V)}
Let $V_1$ and $V_2$ be $\mb{Q}[F]$-modules of finite dimension over $\mb{Q}$. If $V_1\otimes_\mb{Q}\mb{C}\cong V_2\otimes_\mb{Q}\mb{C}$ are isomorphic as $\mb{C}[F]$-modules, then $\alpha_{F\curvearrowright V_1}^*=\alpha_{F\curvearrowright V_2}^*$.
\end{corollary}

Now, let $G$ be a finitely generated virtually abelian group with torsion-free Fitting subgroup $T\lhd G$ (hence $T\cong \mb{Z}^h$ for some $h$), and let $S:1\to T\to G\xrightarrow{\pi} F\to 1$ be the associated group extension. 
Set $V:=T\otimes_\mb{Z}\mb{Q}$, which is naturally a $\mb{Q}[F]$-module. Note that if $E\lhd F$ is a normal subgroup, then the $0$-homology $H_0(E,V)$ of $E$ with coefficients in $V$ is also a $\mb{Q}[F]$-module.
\begin{proposition}[{\cite[Sec.\ 2]{dSMS}}]\label{main result of dSMS}
For each subgroup $E\leq F$, the series $\zeta_{S,E}^{\s}(s)$ and $\zeta_{E\curvearrowright V}^\s(s)$ have the same local factor at $p$ for almost all primes $p$.  For each normal subgroup $E\lhd F$, the series $\zeta_{S,E}^{\n}(s)$ and $\zeta_{F\curvearrowright H_0(E,V)}^\n(s)$ have the same local factor at $p$ for almost all primes $p$. 
\end{proposition}

\begin{corollary}\label{the abscissa of convergence of virtually abelian groups}
For each $E\leq F$ we have $\alpha_{S,E}^{\s}=\alpha_{E \curvearrowright V}^\s$, and for each $E\lhd F$ we have $\alpha_{S,E}^{\n}=\alpha_{F\curvearrowright H_0(E,V)}^\n$.
 \end{corollary}
 \begin{proof}
By Theorem \ref{main resulta of Su} and Theorem \ref{basic properties of cone integrals}(2), we can disregard a finite number of local factors in the computation of $\alpha_{S,E}^*$. By the definition of  $\zeta_{E\curvearrowright V}^\s(s)$ and $\zeta_{F\curvearrowright H_0(E,V)}^\n(s)$, we can also disregard a finite number of local factors in the computation of $\alpha_{E \curvearrowright V}^\s$ or $\alpha_{F\curvearrowright H_0(E,V)}^\n$. Thus, the corollary follows from Proposition \ref{main result of dSMS}. 
 \end{proof}
 \begin{corollary}
Either $\alpha_G^\s=\alpha_T^\s$ or else $\alpha_G^\s=\alpha_T^\s+1$. Moreover, the latter occurs if and only if some element of $F$  acts as minus the identity on $T$. 
 \end{corollary}
 \begin{proof}
We have $\alpha_G^\s=\max\{\alpha_{S,E}^\s: E\leq F\}$ and we know that $\alpha_G^\s\geq\alpha_T^\s=h$, where $h$ is the rank of $T$.
By Corollary \ref{the abscissa of convergence of virtually abelian groups} and Remark \ref{remark on the abscissa for virtually abelian groups=h+1} we find that $\alpha_{S,E}^\s$ is either $h+1$ or $\alpha_{S,E}^\s\leq h$. This proves that $\alpha_G^\s$ is either $h$ or $h+1$. If $\alpha_G^\s=h+1$, then $\alpha_{S,E}^\s=h+1$ for some non-trivial $E\leq F$, and by Remark \ref{remark on the abscissa for virtually abelian groups=h+1} there is a non-trivial element of $E$ that acts on $T$ as minus the identity. Conversely, if there is $f\in F$ that acts on $T$ as minus the identity, then we can replace $f$ by some power and assume that $f^2=1$. If $E=\langle f\rangle$, then $\alpha_{S,E}^\s=h+1$ by Remark \ref{remark on the abscissa for virtually abelian groups=h+1}.  
 \end{proof} 
\begin{proposition}\label{theorem for virtually abelian groups}
For each $i\in\{1,2\}$, let $G_i$ be a finitely generated virtually abelian group with torsion-free Fitting subgroup $T_i$, and let $S_i:1\to T_i\to G_i\rightarrow F_i\to 1$ be the induced group extension. 
Assume that there is a $\mb{C}$-linear isomorphism $\alpha:T_1\otimes_\mb{Z}\mb{C}\to T_2\otimes_\mb{Z} \mb{C}$ and a group isomorphism $\gamma:P_1\to P_2$ such that $\alpha(f\cdot v)=\gamma(f)\cdot \alpha(v)$ for all $v\in T_1\otimes_\mb{Z}\mb{C}$ and $f\in F_1$. Then for each $E_1\leq F_1$ it holds that $\alpha_{S_1,E_1}^{\s}=\alpha_{S_2,\gamma(E_1)}^{\s}$, and for each $E_1\lhd F_1$ it holds that $\alpha_{S_1,E_1}^{\n}=\alpha_{S_2,\gamma(E_1)}^\n$. In particular, $\alpha_{G_1}^\s=\alpha_{G_2}^\s$ and $\alpha_{G_1}^\n=\alpha_{G_2}^\n$.
    \end{proposition}
\begin{proof}
This follows from Corollary \ref{the abscissa of convergence of virtually abelian groups} and Corollary \ref{coro: invariance of a(F,V)}.
\end{proof}

\medskip

\subsection{The $R$-Mal'cev completion for virtually nilpotent groups}
Proposition \ref{theorem for virtually abelian groups} is our version of Theorem \ref{main theorem} for virtually abelian groups, and we want to formulate a similar result for virtually nilpotent groups. To do this we recall the notion of Mal'cev completion for virtually nilpotent groups \cite[Section 1]{Su}.
We begin by reviewing the definition of nilpotent $R$-powered groups.
\begin{definition}
Let $c\in\mb{N}$. A commutative ring $R$ is said to be {\em $c$-binomial} if $R\to R\otimes_\mb{Z}\mb{Q}$ is injective and if $\binom{r}{k}:=\frac{r(r-1)\cdots (r-k+1)}{k!}$ belongs to $R$ for all $r\in R$ and $k=1,\ldots,c$. For such a ring, a nilpotent group $N$ of nilpotency class $\leq c$ is said to be $R$-powered if for all $r\in R$ and $n\in N$, an element $n^r\in N$ has been defined such that the following holds:
\begin{enumerate}[(i)]
\item $n^1=n$, $n^{r_1+r_2}=n^{r_1}n^{r_2}$, $(n^{r_1})^{r_2}=n^{r_1r_2}$ for all $n\in N$, $r_1,r_2\in R$.
\item $m^{-1}n^rm=(m^{-1}nm)^r$ for all $m,n\in N$, $r\in R$.
\item The Hall-Petresco formula holds for all $k$-tuples $(n_1,\ldots,n_k)$ of elements of $N$ and all $r\in R$ \cite[Chap., 6]{War}  .
\end{enumerate}
Note that (iii) makes sense by \cite[Theorem 6.1]{War} since $N$ has nilpotency class $\leq c$ and therefore only the first $c$ binomials $\binom{r}{1},\ldots\binom{r}{c}$ appear in the formula.

A morphism $\varphi:N\to M$ of nilpotent $R$-powered groups of nilpotency class $\leq c$ is a group homomorphism such that $\varphi(n^r)=\varphi(n)^r$ for all $n\in N$ and $r\in R$. They will be called $R$-morphisms.
\end{definition}
\begin{definition}
Let $N$ be a $\mf{T}$-group, say of nilpotency class $c$. Let $R$ be a $c$-binomial ring.
The $R$-Mal'cev completion of $N$ is a nilpotent $R$-powered group $N^R$ 
(necessarily of the same nilpotency class as $N$) together with a homomorphism $\iota:N\to N^R$ satisfying the following universal property: if $M$ is another nilpotent $R$-powered group of nilpotency class $\leq c$ and $\varphi:N\to M$ is a group homomorphism, then there exists a unique $R$-morphism $\tilde{\varphi}:N^R\to M$ such that $\tilde{\varphi}\circ\iota=\varphi$.
\end{definition}
The theory of nilpotent $R$-powered groups, in particular the proof of the existence of the $R$-Mal'cev completion for $\mf{T}$-groups, is expounded in \cite[Chapters 10 and 11]{War} under the assumption that $R$ is a binomial domain (i.e.\ $\binom{r}{k}\in R$ for all $r\in R$ and all $k\in\mb{N}$). However,  everything can be extended without further modifications to $c$-binomial rings in the case of nilpotency class $\leq c$.  
\begin{remark}
The unipotent group scheme $\mf{N}$ over $\mb{Q}$ defined by a $\mf{T}$-group $N$ is precisely the group scheme that represents the functor $K\mapsto N^K$ from commutative $\mb{Q}$-algebras to groups. If $\mf{N}_1$ and $\mf{N}_2$ are the unipotent group schemes over $\mb{Q}$ defined respectively by two $\mf{T}$-groups $N_1$ and $N_2$, then $\mf{N}_1$ and $\mf{N}_2$ are isomorphic after base change with a field $K\supset\mb{Q}$ if and only if $N_1^K$ and $N_2^K$ are isomorphic as nilpotent $K$-powered groups.
\end{remark}
\begin{definition}
Given $c\in\mb{N}$ and a $c$-binomial ring $R$, we define a category $\mc{V}_{c,R}$ as follows. The objects are group extensions $S:1\to N\to  G\to F\to 1$, where $N$ is a nilpotent $R$-powered group of nilpotency class $\leq c$, $F$ is a finite group, and it is required that for any $g\in G$, the automorphism of $N$ induced by conjugation by $g$ is an $R$-automorphism.
The morphisms in $\mc{V}_{c,R}$ are morphisms of short exact sequences of groups $(u,v,w):{S}\to {S'}$ such that $u$ is an $R$-morphism. We also call them $R$-morphisms.
\end{definition} 
\begin{definition} 
Let $S:1\to N\xrightarrow{\iota} G\xrightarrow{\pi} F\to 1$ be an object of $\mc{V}_{c,\mb{Z}}$, where $N$ is a $\mf{T}$-group, and let $R$ be a $c$-binomial ring. The $R$-Mal'cev completion of $S$ is an object $S^R$ of $\mc{V}_{c,R}$ together with a morphism of short exact sequences $(i,j,k):S\to S^R$ satisfying the following universal property: if $(u,v,w):S\to T$ is a morphism of short exact sequences, where $T$ is an object of $\mc{V}_{c,R}$, then there exists a unique $R$-morphism $(\tilde{u},\tilde{v},\tilde{w}):S^R\to T$ such that $(\tilde{u},\tilde{v},\tilde{w})\circ (i,j,k)=(u,v,w)$.
\end{definition}  
The following construction of the $R$-Mal'cev completion of $S$ was given in \cite[Section 1]{Su} under the assumption that $R$ is binomial. However, everything remains valid in our situation. We may assume that $\iota$ is an inclusion and that $\pi$ is a quotient map. Let $s:F\to G$ be a section of $\pi$ (i.e.\ $\pi\circ s=\on{id}_F$) such that $s(1)=1$ and $s(f^{-1})=s(f)^{-1}$ for all $f\in F$. Then there are maps $\sigma:F\to \on{Aut}(N)$ and $\psi:F\times F\to N$ such that for all $f,f'\in F$ and $n\in N$ we have 
$\sigma(f)(n)=s(f)ns(f)^{-1}$ and $s(f)s(f')=\psi(f,f')s(ff')$.
The pair $(\sigma,\psi)$ is called {\em the cocycle associated to $S$ and the section $s$}. It satisfies the following cocycle conditions:
\begin{align}\label{cocycle conditions}
\sigma(f)\sigma(f')&=\mu(\psi(f,f'))\sigma(ff')\quad\forall f,f'\in F\\
\nonumber \psi(f,f')\psi(ff',f'')&=\sigma(f)(\psi(f,f'))\psi(f,f'f'')\quad\forall f,f',f''\in F.
\end{align}
The group $G$ can be identified with the group $N\times_{(\sigma,\psi)} F$ whose underlying set is $N\times F$ and where the operations are given by
\begin{align}\label{operations in NxF}
(n,f)\cdot (n',f')=(n\sigma(f)(n')\psi(f,f'),ff').
\end{align}
Under this identification, $N$ becomes $N\times \{1\}$. 

We now consider the $R$-Mal'cev completion $N\hookrightarrow N^R$ of $N$. Note that $\sigma$ can be extended to a map $F\to\on{Aut}_R(N)$, and $\psi$ can be seen as a map $F\times F\to N^R$. We still denote these extensions by $\sigma$ and $\psi$. We obtain a group $N^R\times_{(\sigma,\psi)} F$ whose underlying set is $N^R\times F$ and where the operations are given by (\ref{operations in NxF}). Indeed, we get a group extension $1\to N^R\to N^R\times_{(\sigma,\psi)} F\to F\to 1$, which is an object in $\mc{V}_{c,R}$. This extension together with the inclusions $N\hookrightarrow N^R$, $N\times_{(\sigma,\psi)} F\hookrightarrow N^R\times_{(\sigma,\psi)} F$, $F=F$, is in fact the $R$-Mal'cev completion $S^R$ of $S$.
\begin{remark}\label{description of an iso between exact sequences}
For $i\in\{1,2\}$ let $G_i$ be a finitely generated virtually nilpotent group with torsion-free Fitting subgroup $N_i$, and let $S_i:1\to N_i\to G_i\to F_i\to 1$ be the associated group extension. We may assume that $G_i=N_i\times_{(\sigma_i,\psi_i)} F_i$, with $(\sigma_i,\psi_i)$ satisfying the cocycle conditions (\ref{cocycle conditions}). Let $c\in\mb{N}$ be an upper bound for the nilpotency classes of $N_1$ and $N_2$, and let $R$ be a $c$-binomial domain. We describe what an $R$-isomorphism between $S_1^R$ and $S_2^R$ is. By definition this is a morphism of exact sequences $(u,v,w):S_1^R\to S_2^R$, where $u$ is an $R$-isomorphism and $w:F_1\to F_2$ is an isomorphism. Note that $v$ has the form
\begin{align}\label{form of beta}
v(n,f)=(u(n)\tau(f),w(f))
\end{align}
for some map $\tau:F_1\to N_2^R$. One can easily check that given an $R$-isomorphism $u:N_1^R\to N_2^R$, a group isomorphism $w:F_1\to F_2$, and a map $\tau:F_1\to N_2^R$, if we define $v$ as in (\ref{form of beta}), then $(u,v,w)$ is an $R$-isomorphism between $S_1^R$ and $S_2^R$ if and only if for all $f,f'\in F$ and  $n\in N$, 
\begin{align}\label{condition on tau}
u(\sigma_1(f)(n))u(\psi(f,f'))\tau(ff')=\tau(f)\sigma_2(w(f))(u(n))\sigma_2(w(f))(\tau(f'))\psi_2(w(f),w(f')).
\end{align}
\end{remark}
\begin{remark}
Let $G$ be a finitely generated virtually nilpotent group with torsion-free Fitting subgroup $N$, and let $S:1\to N\to G\to F\to 1$ be the associated group extension. 
We may assume that $G=N\times_{(\sigma,\psi)} F$, with $(\sigma,\psi)$ satisfying the cocycle condition (\ref{cocycle conditions}). For each $\mb{Q}$-commutative algebra $K$ we define $\mf{G}(K):=N^K\times_{(\sigma,\psi)} F$. Then $\mf{G}$ is an affine group scheme over $\mb{Q}$ isomorphic as scheme to $\dot\cup_{f\in F}\mb{A}_\mb{Q}^{h}$, where $h:=h(N)$. It has $\mf{N}$, the unipotent group scheme defined by $N$, as its connected component. It is easy to check that $\mf{G}$ is well-defined up to $\mb{Q}$-isomorphism. 
We call $\mf{G}$ the group scheme defined by $G$.  
Let $G'$ be another finitely generated virtually nilpotent group with torsion-free Fitting subgroup $N'$, let $S'$ be the associated group extension, and let $\mf{G}'$ be the group scheme over $\mb{Q}$ defined by $G'$. It is not difficult to show that $\mf{G}$ and $\mf{G'}$ are isomorphic after base change with a field $K\supset\mb{Q}$ if and only if $S_1^K$ and $S_2^K$ are $K$-isomorphic.
\end{remark}
\medskip

\subsection{Theorem \ref{main theorem} for virtually nilpotent groups}
We can now state the analogous of Theorem \ref{main theorem} for virtually nilpotent groups, which also extends Proposition \ref{theorem for virtually abelian groups}. 
\begin{theorem}\label{main theorem for virtually nilpotent}
Let $G_i$ be a finitely generated virtually nilpotent group with torsion-free Fitting subgroup $N_i$ and let $S_i:1\to N_i\to G_i\to F_i\to 1$ be the induced group extension for $i=1,2$. Assume that there is a $\mb{C}$-isomorphism $(u,v,w):S_1^{\mb{C}}\cong S_2^{\mb{C}}$. Then for each $E_1\leq F_1$ we have $\alpha_{S_1,E_1}^{\s}=\alpha_{S_2,w(E_1)}^{\s}$, and for each $E_1\lhd F_1$ we have $\alpha_{S_1,E_1}^{\n}=\alpha_{S_2,w(E_1)}^{\n}$.
\end{theorem}
The rest of the section is devoted to the proof of this theorem. The idea is similar to that of the proof of Theorem \ref{main theorem: general version} given in Section \ref{Sec: Proof of main theorem for rings}.
\begin{lemma}\label{algebraic geometry lemma 2}
For $i\in\{1,2\}$, let $G_i$ be a finitely generated virtually nilpotent group with torsion-free Fitting subgroup $N_i$, and let $S_i:1\to N_i\to G_i\to F_i\to 1$ be the associated group extension. Assume that $F_1\cong F_2$ and let $w_0:F_1\to F_2$ be an isomorphism.
If for some field extension $K\supset \mb{Q}$  there is $K$-isomorphism $(u,v,w):S_1^K\cong S_2^K$ with $w=w_0$, then this also holds for some number field.
\end{lemma} 
\begin{proof}
We can assume that $G_i=N_i\times_{(\sigma_i,\psi_i)} F_i$, with $(\sigma_i,\psi_i)$ satisfying the cocycle condition (\ref{cocycle conditions}). 
Let $I_{w_0}$ be the functor from the category of commutative algebras over $\mb{Q}$ to the category of sets that is defined by:
\begin{align}\label{definition of I_w}
I_{w_0}(K)=\{(u,\tau)\ | \ u:N_1^K\to N_2^K \mbox{ is a } K\mbox{-isomorphism and}\ \tau:F_1\to N_2^K\ \mbox{satisfies  (\ref{condition on tau})}\}.
\end{align}
By using the $\log$ isomorphism between the unipotent group scheme over $\mb{Q}$ defined by $N_i$ and its Lie algebra (viewed as functor on commutative algebras over $\mb{Q}$) and the fact that, owing to nilpotency, the Baker-Campbell-Hausdorff formula is finite, 
 one can easily show that $I_{w_0}$ is represented by a closed subscheme of $\on{GL}_h\times\prod_{f\in F_1}\mb{A}^h\times\on{Spec}(\mb{Q})$, where $h:=h(N_1)=h(N_2)$. The hypothesis implies that  $I_{w_0}$ is not the empty scheme by Remark \ref{description of an iso between exact sequences}. If $K$ is the residue field at a closed point of $I_{w_0}$, then $K$ is a number field and $I_{w_0}(K)\neq\emptyset$. Again by Remark \ref{description of an iso between exact sequences}, a pair $(u,\tau)\in I_{w_0}(K)$ yields a $K$-isomorphism $(u,v,w):S_1^K\to S_2^K$, with $v$ defined by (\ref{form of beta}) and $w=w_0$.
\end{proof} 
 \begin{definition}
Let $R$ be a $c$-binomial ring, and let ${S}: 1\to N\to G\xrightarrow{\pi} F\to 1$ be an object of $\mc{V}_{c,R}$. For each subgroup $E\leq F$ we define the following formal series:
\begin{align}
\zeta_{{S},E}^\s(s):=\sum_{\substack{A\leq_f G:\\
 \pi(A)=E\ \wedge\ A\cap N\leq_R N}}[\pi^{-1}(E):A]^{-s},\quad \zeta_{{S},E}^\n(s):=\sum_{\substack{A\lhd_f G:\\
 \pi(A)=E\ \wedge\ A\cap N\leq_R N}}[\pi^{-1}(E):A]^{-s},
\end{align}
where the notation $A\cap N\leq_R N$ means that $A\cap N$ is an $R$-subgroup of $N$, that is, a subgroup such that $n^r\in A\cap N$ for all $n\in A\cap N$ and $r\in R$. 
 \end{definition}
 \begin{proposition}\label{cone integral data for virtually nilpotent groups}
Let $G$ be a finitely generated virtually nilpotent group with torsion-free Fitting subgroup $N$, say of nilpotency class $c$ and Hirsch length $h$, and let ${S}:1\to N\to G\xrightarrow{\pi} F\to 1$ be the associated extension. Let $*\in\{\leq,\lhd\}$.
Then for each $E\leq F$, with $E$ normal if $*=\lhd$, there exists a cone integral data $\mc{D}_E^*$ over $\mb{Q}$ such that the following holds.
\begin{enumerate}
\item For each prime $p$ we have $\zeta_{{S}^{\mb{Z}_p},E}^*(s)=(1-p^{-1})^h Z_{\mc{D}_E^*}(s-h-|E|+1,p)$.
\item Let $K$ be a number field and $\mc{O}$ its ring of integers. For each maximal ideal $\p\subset\mc{O}$ for which $\mc{O}_\p$ is $c$-binomial (e.g. $c!\notin \p$) we have $\zeta_{{S}^{{\mc{O}}_\p},E}^*(s)=(1-\mathbf{N}\p^{-1})^{-h}Z_{\mc{D}_E^*\otimes_\mb{Q}K}(s-h-|E|+1,\p)$, where $\mc{D}_E^*\otimes_\mb{Q}K$ denotes the collection $\mc{D}_E^*$ viewed as cone integral data over $K$.
\end{enumerate}
\end{proposition}
\begin{proof}
A construction of a collection $\mc{D}_E^*$ that satisfies (1) was obtained in \cite[Section 2.2]{Su}.
With the same proof one can show that this collection also satisfies (2). 
\end{proof}
\medskip

\noindent {\em Proof of Theorem \ref{main theorem for virtually nilpotent}:}
 Lemma \ref{algebraic geometry lemma 2} enables us to replace $\mb{C}$ by a number field $K$ in the hypothesis of the theorem without modifying $w$.
We can assume that $G_i=N_i\times_{(\sigma_i,\psi_i)} F_i$, with $(\sigma_i,\psi_i)$ satisfying the cocycle condition (\ref{cocycle conditions}). Let $\tau: F_1\to N_2^K$ be the map of Remark \ref{description of an iso between exact sequences} defined from $(u,v,w)$. 
Let $\{x_1,\ldots,x_h\}$ be a Mal'cev basis for $N_1$ and $\{y_1,\ldots,y_h\}$ a Mal'cev basis for $N_2$.

We denote by $T$ the set of maximal ideals $\p\subset\mc{O}:=\mc{O}_K$ satisfying the following conditions:
\begin{enumerate}
\item ${\mc{O}}_\p$ is $c$-binomial (e.g. $c!\notin \p$) so that $N_1^{{\mc{O}}_\p}$ and $N_2^{{\mc{O}_\p}}$ are ${\mc{O}}_\p$-powered groups;
\item $u(x_i), \tau(f)\in N_2^{{\mc{O}}_\p}$ for all $i=1,\ldots,h$ and $f\in F_1$;
\item $u^{-1}(y_i)\in N_1^{\mc{O}_\p}$ for all $i=1,\ldots,h$;
\end{enumerate}
Note that almost all maximal ideals $\p\subset\mc{O}$ are in $T$. For such a $\p$, conditions (2) and (3) imply that $u$ induces an ${\mc{O}}_\p$-isomorphism $N_1^{\mc{O}_\p}\to N_2^{\mc{O}_\p}$. According to Remark \ref{description of an iso between exact sequences}, $(u,v,w)$ induces an isomorphism $S_1^{{\mc{O}}_\p}\to S_2^{{\mc{O}}_\p}$. 

Fix $E_1\leq F_1$, with $E_1$ normal if $*=\lhd$, and set $E_2:=w(E_1)$. We consider the cone integral data $\mc{D}_{E_i}^*$ of Proposition \ref{cone integral data for virtually nilpotent groups} applied to $S_i$. 
From the above paragraph we deduce that $Z_{\mc{D}_{E_1}^*\otimes K}(s-h-|E_1|+1,\p)=Z_{\mc{D}_{E_2}^*\otimes K}(s-h-|E_2|+1,\p)$ for all almost all maximal ideals $\p\subset\mc{O}$, where $\mc{D}_{E_i}^*\otimes K$ denotes the same collection $\mc{D}_{E_i}^*$ viewed as cone integral data over $K$. Therefore, by Corollary \ref{coro:basic properties of cone integrals}(2), $Z_{\mc{D}_{E_1}^*}(s-h-|E_1|+1)$ and $Z_{\mc{D}_{E_2}^*}(s-h-|E_2|+1)$ have the same abscissa of convergence. Finally, by Proposition \ref{cone integral data for virtually nilpotent groups}, $\zeta_{S_1,E_1}^*(s)$ and $\zeta_{S_2,E_2}^*(s)$ have the same abscissa of convergence. \qed

\bibliography{References}
\bibliographystyle{abbrv}
\end{document}